\documentclass{siamart0516}
\usepackage[letterpaper,margin=1in]{geometry}

\usepackage{lipsum}
\usepackage{amsfonts}
\usepackage{graphicx}
\usepackage{caption}
\usepackage{subcaption}
\usepackage{algorithmic}
\usepackage{amsmath}     
\usepackage{amssymb}
\usepackage{mathtools}
\usepackage{bbm}
\usepackage{color}
\usepackage{siunitx}
\usepackage{hyperref}
\usepackage{cleveref}

\usepackage{appendix}
\usepackage{algorithm}
\usepackage{algorithmic} 

\def\DT{h}        
\def\MUINI{\mu_0}   
\def\TR{\mathrm{T}}  

\renewcommand{\epsilon}{\varepsilon}
\renewcommand{\leq}{\leqslant}
\renewcommand{\geq}{\geqslant}

\newtheorem{remark}{Remark}[section]
\newtheorem{assumption}{Assumption}

\newcommand{\TheTitle}{{Convergence of the likelihood ratio method for linear response of non-equilibrium stationary states} } 
\newcommand{\TheAuthors}{P. Plech\'{a}\v{c}, G. Stoltz and T. Wang}

\headers{Linear response of non-equilibrium stationary states}{\TheAuthors}

\title{\TheTitle}

\author{
Petr Plech\'{a}\v{c}\thanks{University of Delaware, Newark, DE, 19716 {(\email{plechac@math.udel.edu})}}
\and
Gabriel Stoltz\thanks{Universit\'{e} Paris-Est, CERMICS (ENPC), INRIA
(\email{gabriel.stoltz@enpc.fr})} 
\and
Ting Wang\thanks{U.S. Army Research Laboratory, Aberdeen Proving Ground, MD, 21005   (\email{tingw@udel.edu})}
}

\begin{document}

\maketitle

\begin{abstract}
We consider numerical schemes for computing the linear response of 
steady-state averages of stochastic dynamics with respect to a perturbation of the drift part of the stochastic differential equation. 
The schemes are based on Girsanov's change-of-measure theory to reweight trajectories with factors derived from a linearization of the Girsanov weights.
We investigate both the discretization error and the finite time approximation error.
The designed numerical schemes are shown to be of bounded variance with respect to the integration time, which is a desirable feature for long time simulation. We also show how the discretization error can be improved to second order accuracy in the time step by modifying the weight process in an appropriate way.    
\end{abstract}




\section{Introduction}

In many applications one is interested in knowing the response of the steady-state distribution of a stochastic dynamical system with respect to a perturbation to the dynamics.
For example, an important quantity of interest in the linear response theory of statistical mechanics is the transport coefficient $\rho$ that relates the average response of the system in its steady state to the external forcing applied to the system \cite{evans2008statistical, tuckerman2010statistical}, e.g., the mobility, the shear viscosity and the thermal conductivity.
The problem admits a simple mathematical interpretation:
given the stochastic dynamics $X(t)$ with its invariant measure $\mu$ and the perturbed dynamics $X^{\epsilon}(t)$ with its invariant measure $\mu^{\epsilon}$, how does the perturbed steady-state average of some observable $\theta$, i.e.,  $\mu^{\epsilon}(\theta) \triangleq \int \theta(x) \, \mu^{\epsilon}(dx)$,
react to the perturbation with magnitude $\epsilon \in \mathbb{R}$?
That is, we are interested in computing the derivative
\begin{equation}\label{eqn:transport-coefficient}
\rho(\theta) \triangleq \frac{d}{d\epsilon} \mu^{\epsilon}(\theta) = \lim_{\epsilon \to 0} \frac{1}{\epsilon}\Big(\mu^{\epsilon}(\theta) - \mu(\theta)\Big).
\end{equation}

Numerically, due to possible high dimensionality, the averages with respect to
the invariant measure are often approximated as ergodic averages with very long integration times.
Traditional numerical approaches for computing transport coefficients can be classified into two main categories: (i)~either through reformulating the linear response as an integrated correlation based on the Green-Kubo formula, or (ii)~through approximating the derivative in~\eqref{eqn:transport-coefficient} using finite differences.
See, for example, \cite[Chapter~5]{lelievre2016partial} for a review.
Let us also mention another technique proposed recently in \cite{assaraf2018computation}, based on coupling $X(t)$ with its associated tangent process.   

The estimation of \eqref{eqn:transport-coefficient} is also known as the steady-state sensitivity analysis in the stochastic simulation community \cite{asmussen2007stochastic, glasserman2013monte, rubinstein2016simulation}.
The likelihood ratio (LR) method is one of the most widely used methods for sensitivity analysis in this community \cite{glynn1990likelihood}. However, the method is rarely used for steady-state sensitivity analysis since it is always numerically observed that the variance grows rapidly in terms of the integration time. 
Nevertheless, thanks to the zero mean martingale structure of the stochastic exponential involved in the estimator, the large variance issue of the LR method can be remedied by centering the estimator at the steady-state average $\mu(\theta)$ \cite{asmussen2007stochastic, glynn2019likelihood, wang2019steady, wang2016efficiency}. This simple idea leads to the centered LR (CLR) method.
In particular, it has been theoretically shown for continuous time jump Markov processes that the variance of the CLR estimator is uniformly bounded in terms of integration time \cite{wang2019steady}, which suggests that CLR is particularly useful for steady-state sensitivity analysis.

The aim of this work is to introduce the CLR method for sensitivity analysis of stochastic differential equations (SDEs).
In other words, we propose an alternative numerical approach based on CLR for computing the linear response $\rho(\theta)$. This method is in fact reminiscent of Bismut's approach to Malliavin calculus~\cite{bismut}. 
Similar to other Monte Carlo numerical approaches, 
there are two sources of errors associated with using CLR to approximate the linear response : (i) a systematic bias, which stems from the finite time step $\DT$ used to discretize the continuous dynamics; (ii) a statistical error arising from the finite time $T$ that the dynamics is integrated up to. 
We provide convergence results of the CLR scheme in the weak sense, where both sources of error are taken into account.
Furthermore, by modifying the weight process associated with the CLR estimator appropriately, we introduce a second order CLR estimator that reduces the systematic bias to $\mathcal{O}(\DT^2)$.
The variances of both estimators remain bounded with respect to the integration time and hence is particularly efficient for sampling the linear response of non-equilibrium stationary states.

Our main theoretical tools for analysis are the continuous time Poisson equation associated with the continuous dynamics driven by the underlying SDE and the discrete time Poisson equation associated with the discrete dynamics driven by the Markov chain generated by the numerical discretization.
The main advantage of using Poisson equations is that they serve as a natural link between asymptotic time averages of Markov processes and differential equations.

We do not try to address the most general setting in this paper. 
We assume in fact that the state space is compact and that the diffusion term is non-degenerate. 
This simplifying setting allows us to avoid 
some technical difficulties in the proofs so that 
we can focus on the design of the numerical schemes. 
Although the current setting excludes some important applications such as the linear response estimation 
for hypoelliptic systems (e.g., underdamped Langevin dynamics), we emphasize that the results in this work can be proven in more general settings under additional assumptions. For instance, for underdamped Langevin dynamics, we can introduce a sequence of smooth bounding functions to handle unbounded states and obtain estimates on solutions of Poisson equations and their derivatives \cite{kopec2015weak, leimkuhler2015computation}, so that the same arguments as in the proofs of the current work apply.

The manuscript is organized as follows. In Section~\ref{sec:continuous-dynamics}, we set up the probabilistic framework for the linear response problem, and provide some preliminaries regarding the steady-state LR method. Specifically, we show that the continuous time CLR estimator is asymptotically unbiased and is of uniformly bounded variance with respect to the integration time $T$. 
The weak numerical schemes that discretize the SDE and the ergodicity of the associated Markov chains are discussed in Section~\ref{sec:discrete-dynamics}. 
Section~\ref{sec:CLR-scheme} contains the main results of this work. We propose a weak first order CLR scheme and analyze both the bias and variance of the associated estimator.
The order of weak error is improved in Section~\ref{sec:second-order-scheme} where we design a weak second order CLR estimator for a specific second order discretization scheme by modifying the weight process in an appropriate way.
The strategy is generalized in Section~\ref{sec:generel-2nd-order}.
Our theoretical results are illustrated by a numerical example in Section~\ref{sec:numerical-results}.
Some technical results are gathered in Section~\ref{sec:proofs}.
Finally, we comment that for the ease of presentation, most of the proofs in this work are presented in the scalar setting although the results are stated in the multi-dimensional setting.


\section{Continuous time estimator of the linear response}\label{sec:continuous-dynamics}

\subsection{Linear response for non-equilibrium dynamics}

We study dynamics whose evolution is dictated by a stochastic differential equation. Given a probability space $(\Omega, \mathcal{F}, \mathbb{P})$,
we denote by $W(t)=(W^1(t), \ldots, W^d(t))^\TR$ the $d$-dimensional standard Brownian motion on this probability space.
We consider the stochastic process $X(t)$ that satisfies a SDE on the following compact state space, to simplify the mathematical analysis.

\begin{assumption}\label{assume:bounded-config-space}
The state space $\mathcal{X}$ is the $d$-dimensional torus $\mathbb{T}^d$ (where $\mathbb{T} = \mathbb{R}/\mathbb{Z}$).
\end{assumption}

More precisely, we consider 
\begin{equation}\label{eqn:SDE}
    dX(t) = b(X(t))\,dt + \sigma(X(t))\, dW(t),
\end{equation}
where
$b:\mathcal{X} \to \mathbb{R}^{d}$ is the drift term and $\sigma: \mathcal{X} \to \mathbb{R}^{d \times d}$ is the diffusion term. We denote the initial distribution of $X(0)$ by $\MUINI$ and by $\mathcal{F}_t$ the natural filtration associated with $X(t)$. We further assume the following conditions on the drift and diffusion terms.

\begin{assumption}\label{assume:PD-diffusion}
The functions $b$ and $\sigma$ are $C^{\infty}$, and the diffusion matrix $\sigma\sigma^\TR$ is positive definite.
\end{assumption}

These conditions guarantee that the SDE~\eqref{eqn:SDE} is non-degenerate and has a unique solution. The solution $X(t)$ of \eqref{eqn:SDE} is a Markov process with infinitesimal generator
\[
\mathcal{L} = b \cdot \nabla + \frac{1}{2}\sigma \sigma^\TR : \nabla^2 = \sum_{i = 1}^d b^i \partial_{i} + \frac{1}{2}\sum_{i = 1}^d \sum_{j = 1}^d \sum_{k = 1}^d \sigma^{ik} \sigma^{jk} \partial_{ij}.
\]
Assumption~\ref{assume:PD-diffusion} ensures that $\mathcal{L}$ is an elliptic differential operator.
For $k = 1, \ldots, d$, we also introduce the operator 
\[
\mathcal{K}^k = \sum_{i=1}^{d}\sigma^{ik} \partial_{i},
\]
so that, for any $C^{\infty}$ test function $\theta: \mathcal{X} \to \mathbb{R}$, the It\^{o} formula reads
\[
\theta(X(t)) = \theta(X(0)) + \int_0^t \mathcal{L} \theta(X(s)) \,ds + \sum_{k=1}^d \int_0^t \mathcal{K}^k \theta(X(s))\, dW^k(s).
\]
In view of Assumptions \ref{assume:bounded-config-space} and \ref{assume:PD-diffusion} the dynamics $X(t)$ admits a unique invariant measure~$\mu$ with a positive density $f$ with respect to the Lebesgue measure, and hence the law of large number holds (see for instance~\cite{bellet2006ergodic, kliemann1987recurrence}):
for any initial state $X(0) \in \mathcal{X}$
and any observable $\theta \in L^1(\mu)$,
\begin{equation}\label{eqn:SDE-ergodicity}
\lim_{t \to \infty}\frac{1}{t} \int_0^t \theta(X(s))\, dt = 
\mu(\theta) \triangleq \int_{\mathcal{X}}\theta(x) \, \mu(dx)   \qquad \mathbb{P}-\mathrm{a.s.}
\end{equation}

Now suppose that there is a small external forcing $F: \mathcal{X} \to \mathbb{R}^d$, typically non-gradient and assumed to be $C^{\infty}$, added to the reference drift. This leads to the following perturbed dynamics:
\begin{equation}\label{eqn:SDE-perturbed}
    dX^{\epsilon}(t) = \Big( b(X^{\epsilon}(t)) + \epsilon F(X^{\epsilon}(t)) \Big) dt + \sigma(X^{\epsilon}(t))\, dW(t).
\end{equation}
The infinitesimal generator of the perturbed dynamics, denoted by  $\mathcal{L}^{\epsilon}$, can be written as 
\[\mathcal{L}^{\epsilon} = \mathcal{L} + \epsilon \widetilde{\mathcal{L}}, \quad \widetilde{\mathcal{L}} = F \cdot \nabla.\]
Similarly to the discussion for the reference dynamics, the perturbed dynamics \eqref{eqn:SDE-perturbed} has a unique solution and admits a unique invariant measure $\mu^{\epsilon}$ with smooth density function $f^{\epsilon}$ with respect to the Lebesgue measure.
For a smooth observable $\theta$,
we are interested in estimating the linear response
\begin{equation*}
    \rho(\theta) = \lim_{\epsilon \to 0}\frac{1}{\epsilon}(\mu^{\epsilon}(\theta) - \mu(\theta))
    = \lim_{\epsilon \to 0}\frac{1}{\epsilon}\int_{\mathcal{X}} \theta(x) (f^{\epsilon}(x) - f(x)) \, dx.
\end{equation*}
In fact, this linear response can be reformulated in terms of the generator $\mathcal{L}$ and the operator $\widetilde{\mathcal{L}}$, using the following result which provides an expansion of $f^{\epsilon}$ in terms of the perturbation magnitude $\epsilon$ (see for instance~\cite[Theorem~5.1]{lelievre2016partial}).
To state it, we introduce the projection operator
\begin{equation}\label{eqn:projection-operator}
\Pi\theta = \theta - \mu(\theta),
\end{equation}
and denote by $L_0^2(\mu) = \Pi L^2(\mu)$ the Hilbert space of square integrable functions with respect to the measure~$\mu$ whose average with respect to $\mu$ is 0.

\begin{theorem}\label{thm:invariant-measure-expansion}
The operator $\Pi \widetilde{\mathcal{L}} \mathcal{L}^{-1}$ is bounded on $L_0^2(\mu)$, and so is its adjoint $(\Pi\widetilde{\mathcal{L}} \mathcal{L}^{-1})^* = (\widetilde{\mathcal{L}} \mathcal{L}^{-1})^*$.
Denoting by $r$ the spectral radius of $(\widetilde{\mathcal{L}}\mathcal{L}^{-1})^*$, i.e., 
\[
r = \lim_{n\to\infty}\left\|\left[\left(\widetilde{\mathcal{L}}\mathcal{L}^{-1}\right)^*\right]^n\right\|_{\mathcal{B}(L^2_0(\mu))}^{1/n},
\]
the invariant probability measure $\mu^{\epsilon}$ can be written, for any $\epsilon < r^{-1}$, 
as $\mu^{\epsilon} = g^{\epsilon}\mu$, where $g^{\epsilon} \in L^2(\mu)$ admits the following expansion in $\epsilon$:
\[
g^{\epsilon} = \left(1+\epsilon \left(\widetilde{\mathcal{L}}\mathcal{L}^{-1}\right)^* \right)^{-1} {\bf{1}} = \left(1 + \sum_{n=1}^{\infty}(-\epsilon)^n \left[\left(\widetilde{\mathcal{L}}\mathcal{L}^{-1}\right)^*\right]^n\right) {\bf{1}}. 
\]
\end{theorem}

A direct result of the above theorem is the following formula for the linear response:
\begin{equation}\label{eqn:rho_theta_first_def}
    \rho(\theta) = -\int_{\mathcal{X}} \widetilde{\mathcal{L}} \mathcal{L}^{-1} \left[\theta(x) - \mu(\theta)\right] \, dx.
\end{equation}

\subsection{The continuous time Poisson equation}\label{sec:continuous_Poisson}

Poisson equations are a useful tool to study asymptotic properties of ergodic Markov processes, in particular to quantify the bias arising from finite time sampling as in \cite{mattingly2010convergence}, and the asymptotic variance of time averages~\cite{Bhattacharya82,meyn2012markov}.
Given a Markov process $X(t)$ with generator $\mathcal{L}$, the Poisson equation associated with a given observable~$\theta$ reads
\begin{equation}\label{eqn:continuous-Poisson-equation}
  -\mathcal{L}\widehat{\theta} = \theta - \mu(\theta).
\end{equation}
We need to provide a functional space guaranteeing the well posedness of this equation. Our analysis requires the solution $\widehat{\theta}$ to be sufficiently regular. We consider the case when $\widehat{\theta} \in C^{\infty}$ to simplify the presentation (although a careful inspection of our proofs shows that only a finite number of derivates are required). This is the case when $\theta \in C^{\infty}(\mathcal{X})$. Indeed, the solution $\widehat{\theta}$ is then well defined (for instance, by considering $\mathcal{L}$ on $L_0^2(\mu)$ and noting that this operator is invertible and has a compact resolvent), and in $C^{\infty}$ by elliptic regularity~\cite{GT01}. For convenience, we denote in the sequel
\[
\mathcal{S} = C^{\infty}(\mathcal{X}), \qquad \mathcal{S}_0 = \Pi\mathcal{S} = \{\theta \in \mathcal{S}~:~ \mu(\theta) = 0\}.
\]

\begin{remark}
Our analysis can be extended to degenerate stochastic dynamics for which the space $\mathcal{S}_0$ is invariant under the operator $\mathcal{L}^{-1}$ (in the sense that, for any $\theta \in \mathcal{S}_0$, it holds $\mathcal{L}^{-1}\theta \in \mathcal{S}_0$). This is the case for instance for dynamics with hypoelliptic generators on compact spaces (as considered in \cite{mattingly2010convergence}), or underdamped Langevin dynamics on bounded or unbounded position spaces \cite{talay2002stochastic,kopec2015weak}, upon changing the definition of~$\mathcal{S}$ to the space of $C^\infty$ functions growing at most polynomially at infinity, and whose derivatives also grow at most polynomially at infinity.
\end{remark}


We are now in position to reformulate the linear response with the solution of the Poisson equation \eqref{eqn:continuous-Poisson-equation}, which is a direct consequence of \eqref{eqn:rho_theta_first_def}.

\begin{proposition}\label{thm:alternative-linear-response-formula}
For any $\theta \in \mathcal{S}$, the linear response $\rho(\theta)$ can be written as 
\[
\rho(\theta) = \int_{\mathcal{X}} F(x)^{\TR} \nabla \widehat{\theta}(x) \,\mu(dx),
\]
where $F(x)^{\TR}$ denotes the transpose of $F(x)$.
\end{proposition}

\subsection{The likelihood ratio method}
We derive the likelihood ratio method for linear response estimation in this section.
Let us denote by $\mathbb{P}$ the path-space probability measure induced by the process $X(t)$.
In view of Assumption~\ref{assume:PD-diffusion}, the vector $\sigma(X(t))^{-1}F(X(t))$ is well defined for all $t>0$.
In the sequel, we denote by $U$ the vector $\sigma^{-1}F$, which is in $\mathcal{S}$ by the above assumptions.
Let us mention that invertibility of $\sigma(X(t))$ is not necessary for the Girsanov change-of-measure theory (as long as there is a smooth function $U$ such that $\sigma U = F$) but we suppose it holds to simplify the mathematical analysis.
We introduce 
\begin{equation*}
L^{\epsilon}(t) = \exp\left(\epsilon\int_0^t U(X(s))\, dW(s) - \frac{\epsilon^2}{2}\int_0^t U(X(s))^{\TR} U(X(s)) \, ds  \right ),
\end{equation*}
and define the measure $\mathbb{P}^{\epsilon}$ such that
\[
\mathbb{P}^{\epsilon}(A) = \int_{A} L^{\epsilon}(t, \omega) \mathbb{P}(d\omega)
\]
for any $\mathcal{F}_t$ measurable set $A$.
For fixed $t > 0$, the Novikov condition is trivially satisfied by our assumptions. 
This implies that $L^{\epsilon}(t)$ is a $\mathcal{F}_t$ martingale with mean~1 and hence
the above defined measure $\mathbb{P}^{\epsilon}$ is a probability measure that coincides with the path-space probability measure of the perturbed process $X^{\epsilon}(t)$ (see for instance \cite{oksendal2013stochastic}).
By the above change-of-measure, we immediately have
\[
\mathbb{E}^{\epsilon}\left\{\frac{1}{t} \int_0^t \theta(X(s)) \, ds\right\} = \mathbb{E}
\left\{ \left(\frac{1}{t} \int_0^t \theta(X(s)) \, ds\right) L^{\epsilon}(t)\right\}.
\]

\begin{remark}
Throughout this paper, the expectation $\mathbb{E}$ and the variance $\mathrm{Var}$ are taken with respect to the initial distribution $\MUINI$ and over all realizations of the reference dynamics \eqref{eqn:SDE}.
\end{remark}

Assuming that we can differentiate with respect to $\epsilon$ inside the expectation $\mathbb{E}$ around $\epsilon = 0$ (see for instance \cite{asmussen2007stochastic, wang2018validity}), it holds
\begin{equation*}
    \frac{d}{d\epsilon}\left[ \mathbb{E}^{\epsilon}\left\{\frac{1}{t}\int_0^t \theta(X(s)) \,ds\right\} \right] = \mathbb{E}\left\{\left(\frac{1}{t}\int_0^t \theta(X(s)) \,ds \right) Z(t)\right\},
\end{equation*}
where 
\begin{equation}\label{eqn:Z}
Z(t) \triangleq \frac{d}{d\epsilon}L^{\epsilon}(t) =  \int_0^t U(X(s))\, dW(s)
\end{equation}
is referred to as the weight process for linear response.
Note that the weight process $Z(t)$ is a zero mean $\mathcal{F}_t$ martingale.
The above derivation suggests using the LR estimator
\[
\left( \frac{1}{t}\int_0^t \theta(X(s))\, ds \right) Z(t)
\]
to approximate the linear response index $\rho(\theta)$, upon choosing $t$ large enough.
As hinted at in the introduction, there exists a simple modification of the LR estimator which consists in centering it around the steady state average $\mu(\theta)$, in order for the variance of this estimator to be bounded. More precisely, we consider the following CLR estimator 
\begin{equation}\label{eqn:continuous-CLR-estimator}
\left( \frac{1}{t}\int_0^t (\theta(X(s)) - \mu(\theta))\, ds\right) Z(t).
\end{equation}
The following theorem states the consistency of CLR estimator.
Its proof demonstrates the interest 
of the Poisson equation \eqref{eqn:continuous-Poisson-equation} in studying the asymptotic limit of time averages.

\begin{theorem}\label{thm:LR-unbiasedness}
For any observable $\theta \in \mathcal{S}$,
\[
\lim_{t \to \infty}\mathbb{E}\left\{ \left(\frac{1}{t}\int_0^t \left(\theta(X(s)) - \mu(\theta)\right)\, ds\right) Z(t)\right\} = \rho(\theta).
\]
\end{theorem}

In fact, convergence rates in terms of inverse powers of~$t$ can be stated, but we refrain from doing so.
We only prove the result for the one-dimensional case $d = 1$. The generalization to the multi-dimensional case is straightforward. 

\begin{proof}
Throughout this proof and the following ones, $C$ is a generic positive constant, which depends only on $\theta$.
In view of the continuous time Poisson equation \eqref{eqn:continuous-Poisson-equation}, the  expectation of the CLR estimator can be rewritten as 
\begin{equation}\label{eqn:continuous-expectation-decomposition}
\begin{aligned}
-\mathbb{E}\left\{\left(\frac{1}{t}\int_0^t\mathcal{L} \widehat{\theta}(X(s)) \, ds\right) Z(t)\right\}
=  - \mathbb{E}\left\{\frac{1}{t} \left[ \widehat{\theta}(X(t)) - \widehat{\theta}(X(0)) \right] Z(t)\right\} & \\
+ \frac{1}{t}\mathbb{E}\left\{ \left[ \widehat{\theta}(X(t)) - \widehat{\theta}(X(0)) - \int_0^t \mathcal{L}\widehat{\theta}(X(s)) \,ds \right] Z(t)  \right\}. & 
\end{aligned}
\end{equation}
By the Cauchy--Schwarz inequality and It\^{o}'s isometry, 
\begin{equation}\label{eqn:bound_Z_order2}
\mathbb{E}\{|Z(t)| \} \leq \mathbb{E}\{Z(t)^2\}^{1/2} = \mathbb{E}\left\{\int_0^t U(X(s))^2\, ds\right\}^{1/2} \leq C \sqrt{t}.
\end{equation}
Since $\widehat{\theta}$ is bounded (see Section~\ref{sec:continuous_Poisson}), 
\[
\mathbb{E}\left\{\frac{1}{t}\left[\widehat{\theta}(X(t)) - \widehat{\theta}(X(0))\right] Z(t)\right\} 
\leq \frac{C}{t} \mathbb{E}\{|Z(t)|\} \leq \frac{C}{\sqrt{t}},
\]
which converges to zero as $t \to \infty$.
Consider now the second term on the right-hand side of~\eqref{eqn:continuous-expectation-decomposition}. By It\^o's formula,
\[
\widehat{\theta}(X(t)) - \widehat{\theta}(X(0)) - \int_0^t \mathcal{L}\widehat{\theta}(X(s)) \,ds 
= 
\int_0^t \mathcal{K}\widehat{\theta}(X(s))\, dW(s),
\]
which is a $\mathcal{F}_t$ martingale.
Recall that $Z(t) = \int_0^t U(X(s))\, dW(s)$ and both $\mathcal{K}\widehat{\theta}(X(s))$ and $U(X(s))$ are square integrable with respect to the product measure $dt \times \mathbb{P}(d\omega)$. Therefore,
\[
\mathbb{E}\left\{\int_0^t \mathcal{K}\widehat{\theta}(X(s))\, dW(s) \int_0^t U(X(s))\, dW(s)\right\}
=
\mathbb{E}\left\{ \int_0^t F(X(s)) \widehat{\theta}^{\prime}(X(s)) \, ds \right\}.
\]
Now, the continuous time Poisson solution $\widehat{\theta}$ is in $\mathcal{S}_0$ (see Section~\ref{sec:continuous_Poisson}), hence the ergodicity of $X(t)$ implies that
\[
\lim_{t\to\infty} \frac{1}{t}\int_0^t F(X(s)) \widehat{\theta}^{\prime}(X(s)) \, ds = \int_{\mathcal{X}} F(x) \widehat{\theta}^{\prime}(x)\, \mu(dx)
\]
almost surely. 
Finally, the desired result follows by dominated convergence and Proposition~\ref{thm:alternative-linear-response-formula}.
\end{proof}

The result roughly says that the average response to a perturbation of the dynamics~\eqref{eqn:SDE} can be computed from the unperturbed dynamics by
re-weighting the observable with the weight process $Z(t)$.
The next result shows that the variance of the CLR estimator remains bounded in terms of the integration time, which is a desirable feature for long time simulation. 
\begin{theorem}\label{thm:LR-variance}
For any observable $\theta \in \mathcal{S}$, there exists a constant $C>0$ such that
\[
\forall t > 0, \qquad
\mathrm{Var}\left\{\left(\frac{1}{t}\int_0^t \left(\theta(X(s)) - \mu(\theta)\right)\, ds\right) Z(t)\right\} \leq C.
\]
\end{theorem}

\begin{proof}
Using the decomposition \eqref{eqn:continuous-expectation-decomposition} as in the proof of the last theorem and the Cauchy-Schwarz inequality, we bound the 
second moment of the CLR estimator by 
\[
\frac{2}{t^2}\mathbb{E}\left\{\left[\widehat{\theta}(X(t)) - \widehat{\theta}(X(0))\right]^2 Z(t)^2\right\}
+ \frac{2}{t^2}\mathbb{E}\left\{\left[\int_0^t \mathcal{K}\widehat{\theta}(X(s))\, dW(s) \right]^2 Z(t)^2\right\}.
\]
In view of \eqref{eqn:bound_Z_order2}, the first term can be 
simply further bounded by $C/t$. It remains to bound the second term. We first apply the Cauchy-Schwarz inequality to obtain
\[
\mathbb{E}\left\{\left[\int_0^t \mathcal{K}\widehat{\theta}(X(s))\, dW(s) \right]^2 Z(t)^2\right\}
\leq
\mathbb{E}\left\{\left[\int_0^t \mathcal{K}\widehat{\theta}(X(s))\, dW(s)\right]^4\right\}^{1/2}\mathbb{E}\left\{Z(t)^4\right\}^{1/2}.
\]
By the Burkholder-Davis-Gundy inequality~\cite{protter2005stochastic},
\[
\mathbb{E}\left\{\left[\int_0^t \mathcal{K}\widehat{\theta}(X(s))\, dW(s)\right]^4\right\}
\leq
C\mathbb{E}\left\{\left(\int_0^t \left[ \mathcal{K}\widehat{\theta}(X(s)) \right]^2 \, ds\right)^2 \right\} \leq Ct^2,
\]
where we have used the fact that $\mathcal{K}\widehat{\theta}$ is uniformly bounded on the state space $\mathcal{X}$.
Similarly, we have
\[
\mathbb{E}\left\{Z(t)^4\right\} \leq C\mathbb{E}\left\{ \left[\int_0^t U(X(s))^2 \, ds\right]^2\right\} \leq C t^2.
\]
Taking the square root of the above estimates and then re-scaling them by $t^2$ leads to the desired bound. 
\end{proof}



\section{The discrete dynamics approximation}\label{sec:discrete-dynamics}
Theorem~\ref{thm:LR-unbiasedness} justifies that both the LR and CLR estimators are asymptotically unbiased. However, in practice, we need to introduce a time step $\DT$ to discretize the continuous dynamics $X(t)$ and obtain a discrete time dynamics $X_n$.
In this section we present a discrete numerical approximation to the continuous dynamics $X(t)$. Furthermore, we establish ergodicity results for the resulting discrete Markov chain.

\subsection{Weak numerical schemes}
A weak numerical scheme that discretizes~$X(t)$ generates a discrete time Markov chain $X_n$ with evolution operator 
\begin{equation}\label{eqn:evolution-operator}
\left(P_{\DT}\theta\right)(x) \triangleq \mathbb{E}_{\DT}\{\theta (X_{n+1}) ~|~ X_n = x\}
\end{equation}
for any $\theta \in \mathcal{S}$.

\begin{remark}
Throughout this paper, in order to alleviate the notation, we denote by $\varphi_n = \varphi(X_n)$ for a given function $\varphi$.
The expectation $\mathbb{E}_{\DT}$ and the variance $\mathrm{Var}_{\DT}$ are taken with respect to the initial distribution $\MUINI$ and over all realizations of the discrete time Markov chain $X_n$ with time step $\DT$.

Furthermore, in order to keep the presentation of calculations in the proofs simple we treat the scalar case ($d=1$) in analysis of the CLR estimator.  We detail the algebraic calculation for the multi-dimensional case in Appendix~\ref{sec:appendix}. In the multi-dimensional case the analysis and proofs generalize directly for the weak first-order CLR estimator. However, in the case of the second-order CLR estimator it is necessary to assume that the noise coefficient $\sigma$ is constant.
\end{remark}

For the ease of exposition, we consider particular weak first and second order schemes that discretize the process $X(t)$. Specifically, we focus on the Euler-Maruyama scheme 
\begin{equation}\label{eqn:first-order-scheme}
    X_{n+1} = X_n + b_n \DT + \sigma_n \Delta W_n,
\end{equation}
for the weak first order scheme, where 
\[
\Delta W_n \stackrel{d}{=} W((n+1)\DT) - 
W(n\DT) \sim \mathcal{N}(0, \DT\mathrm{Id}_d).
\]
For the weak second order scheme, we consider 
\begin{equation}\label{eqn:second-order-scheme}
\begin{split}
    X_{n+1}^{i} = X_n^{i} &+ b_n^i \DT
    + \sum_{k=1}^{d}\sigma_n^{ik} \Delta W_n^k
    + \frac{1}{2}\sum_{k=1}^d(\mathcal{L}\sigma_n^{ik} + \mathcal{K}^k b_n^i) \DT \Delta W_n^k\\
    &+ \frac{1}{2}\sum_{k_1, k_2=1}^d \mathcal{K}^{k_1} \sigma_n^{i k_2} \left(\Delta W_n^{k_1} \Delta W_n^{k_2} + V_n^{k_1k_2}\right)
    + \frac{1}{2}\mathcal{L}b_n^i\DT^2,  \qquad i = 1, \ldots, d
\end{split}
\end{equation}
derived from the second order It\^{o}-Taylor expansion, where $\Delta W_n^k$ is the $k$th component of $\Delta W_n$ and $V_n^{k_1k_2}$ are independent random variables with
\begin{equation}\label{eqn:V-dist}
\begin{split}
&\mathbb{P}(V_n^{k_1k_2} = \pm \DT) = \frac{1}{2}, \qquad k_2 = 1, \ldots, k_1 - 1,\\
&V_n^{k_1k_2} = -\DT, \qquad \qquad k_2 = k_1,\\
&V_n^{k_2k_1} = -V_n^{k_1k_2} \qquad k_2 = k_1 + 1, \ldots, d.
\end{split}
\end{equation}
See, for instance \cite{kloeden2013numerical} for a derivation of the above scheme. 
In the sequel, we denote by $\Phi_{\DT}$ the increment function such that
\[
X_{n+1} = X_n + \Phi_{\DT}(X_n,\Delta W_n,V_n).
\]

\begin{remark}
  In fact, for the first order scheme~\eqref{eqn:first-order-scheme}, there is no need for considering $V_n$ in the argument, and we will therefore simply write $X_{n+1} = X_n + \Phi_{\DT}(X_n,\Delta W_n)$. Actually, we will often write $\Phi_{\DT,n}$ instead of $\Phi_{\DT}(X_n,\Delta W_n)$ to further simplify the notation. 
  We also use the same notation (e.g., $X_n, P_{\DT}, \widehat{\theta}, \Phi_{\DT}$, etc) both for the first and second order schemes.
  The weak order of the corresponding scheme associated with these notations should be clear from the context. 
\end{remark}

Next, we provide the consistency of numerical time-averaging using the above two schemes. 
The proof is essentially the same as that in \cite{mattingly2010convergence}.
The only difference is that our estimates are uniform for a family of smooth functions (typically indexed by the time step $\DT$), which turns out to be crucial for our analysis of the CLR estimator in the next section. The norms $\|\cdot\|_{C^k}$ for $k\geq 1$ are the standard norms associated with the Banach spaces of $C^k$ functions on~$\mathcal{X}$.

\begin{proposition}\label{thm:time-averaging-consistency}
There exists a constant $h^* > 0$ and $C \in \mathbb{R}_+$ such that, for any $\DT \in (0,h^*]$ and any $\varphi \in \mathcal{S}$,
\[
\left|\frac{1}{N}\sum_{n=0}^{N-1} \mathbb{E}_{\DT}\left\{ \varphi_{n}\right\}
- \mu(\varphi)\right| \leq C \|\varphi\|_{C^{2p}}\left(\DT^p + \frac{1}{N\DT}\right),
\]
where $p = 1$ for the first order scheme~\eqref{eqn:first-order-scheme} and $p = 2$ for the second order scheme~\eqref{eqn:second-order-scheme}. 
\end{proposition}

We state the estimate for $\varphi \in \mathcal{S}$ since the functions we will manipulate in the proofs will always belong to the latter functional space, but the above estimate can of course be extended by density to any function in~$C^{2p}$.

\begin{proof}
  We follow the proof of~\cite{mattingly2010convergence}.
  We denote by $C \in \mathbb{R}_+$ a generic constant that may change line by line.  
  We first prove the statement for $p = 1$. Recall also that we write the proof in the one-dimensional setting $d=1$ for simplicity, but it can straightforwardly be extended to spaces of higher dimensions.
  Fix $\varphi \in \mathcal{S}$ and denote by $\widehat{\varphi}$ the solution to the continuous time Poisson equation: 
  \begin{equation}
  \label{eq:continuous_Poisson_equation_varphi}
-\mathcal{L}\widehat{\varphi} = \varphi - \mu(\varphi).
\end{equation}
Recall that $\Phi_{\DT, n} = b_n \DT + \sigma_n \Delta W_n$
for the first order scheme. 
Since $\widehat{\varphi} \in \mathcal{S}_0$ (see Section~\ref{sec:continuous_Poisson}), we can expand 
\[
\widehat{\varphi}_{n+1} = \widehat{\varphi}_{n} + \widehat{\varphi}_{n}^{\prime}\Phi_{\DT, n} 
+ \frac{1}{2} \widehat{\varphi}_{n}^{(2)}\Phi_{\DT, n}^2
+ \frac{1}{6} \widehat{\varphi}_{n}^{(3)}\Phi_{\DT, n}^3 + r_{\widehat{\varphi}}(X_n),
\]
where 
\[
r_{\widehat{\varphi}}(X_n) = \left(\frac{1}{6}\int_0^1 u^3 \widehat{\varphi}^{(4)}(X_n + u \Phi_{\DT, n})\, du\right)\Phi_{\DT, n}^{4}.
\]
Taking expectation of both sides and rearranging terms leads to
\begin{equation*}
\begin{aligned}
\mathbb{E}_{\DT}\left\{\widehat{\varphi}_{n+1}\right\} 
& = \mathbb{E}_{\DT}\left\{\widehat{\varphi}_{n}\right\} 
+ \mathbb{E}_{\DT}\left\{\mathcal{L}\widehat{\varphi}_{n}\right\} \DT
+ \frac{1}{2}\mathbb{E}_{\DT}\left\{\widehat{\varphi}_{n}^{(2)}b_n^2 + \widehat{\varphi}_{n}^{(3)} b_n \sigma_n^2 \right\}\DT^2\\
& \ \ + \frac{1}{6}\mathbb{E}_{\DT}\left\{\widehat{\varphi}_{n}^{(3)} b_n^3 \right\}\DT^3
+ \mathbb{E}_{\DT}\left\{r_{\widehat{\varphi}}(X_n)\right\}.
\end{aligned}
\end{equation*}
Note that, by elliptic regularity~\cite{GT01}, the solution $\widehat{\varphi}_{\DT}$ to~\eqref{eq:continuous_Poisson_equation_varphi} and its derivatives (up to $4$th order here) can be bounded by $C_0 \|\varphi\|_{C^2}$, where $C_0 \in \mathbb{R}_+$ depends on the coefficients $b,\sigma$ in the SDE~\eqref{eqn:SDE} but is independent of~$\varphi$.
Also note that $b$, $\sigma$ and their derivatives are uniformly bounded.
There exists therefore some constant~$C$, independent of $\varphi$, such that
\[
\left| \mathbb{E}_{\DT}\left\{\widehat{\varphi}_{n+1}\right\} - \mathbb{E}_{\DT}\left\{\widehat{\varphi}_{n}\right\} 
- \mathbb{E}_{\DT}\left\{\mathcal{L}\widehat{\varphi}_{n}\right\} \DT \right| \leq C \| \varphi \|_{C^2} \DT^2.
\]
In view of the above inequality and the Poisson equation~\eqref{eq:continuous_Poisson_equation_varphi}, we obtain 
\begin{equation*}
\left| \mathbb{E}_{\DT}\left\{\varphi_{n}\right\} - \mu(\varphi) + \frac{1}{\DT}\mathbb{E}_{\DT}\left\{\widehat{\varphi}_{n+1} - \widehat{\varphi}_{n}\right\}\right| \leq C\|{\varphi}\|_{C^2} \DT.
\end{equation*}
Summing the terms between the absolute values of the above inequalities over $n$ and dividing by~$N$ gives 
\begin{equation}
  \label{eqn:first-order-time-averaging}
\left| \frac{1}{N}\sum_{n=0}^{N-1}\mathbb{E}_{\DT}\left\{\varphi_{n}\right\} - \mu(\varphi) + \frac{1}{N\DT}\mathbb{E}_{\DT}\left\{\widehat{\varphi}_{N} - \widehat{\varphi}_{0}\right\} \right| \leq C\| \varphi \|_{C^2} \DT.
\end{equation}
The desired estimate then follows immediately since $|\widehat{\varphi}_{N} - \widehat{\varphi}_{0}| \leq 2 C_0\| \varphi \|_{C^2}$ (in fact, it is possible to replace $\| \varphi \|_{C^2}$ by $\| \varphi \|_{C^0}$ in the latter inequality).

For the case $p = 2$, an estimate similar to~\eqref{eqn:first-order-time-averaging} holds:
\[
\left| \mathbb{E}_{\DT}\left\{\widehat{\varphi}_{n+1}\right\} - \mathbb{E}_{\DT}\left\{\widehat{\varphi}_{n}\right\} 
- \mathbb{E}_{\DT}\left\{\mathcal{L}\widehat{\varphi}_{n}\right\} \DT
- \frac{1}{2}\mathbb{E}_{\DT}\left\{\mathcal{L}^2\widehat{\varphi}_{n}\right\} \DT^2 \right|
\leq C \| \varphi \|_{C^4} \DT^3,
\]
where the remainder is now bounded by derivatives of~$\varphi$ of order~4 at most (since it involves derivatives of~$\widehat{\varphi}$ of order~6 at most). Combining the above estimate with the Poisson equation leads to
\begin{equation}\label{eqn:second-order-time-averaging}
\left| \frac{1}{N}\sum_{n=0}^{N-1}\mathbb{E}_{\DT}\left\{\varphi_{n}\right\} - \mu(\varphi) + \frac{\DT}{2N}\sum_{n=0}^{N-1}\mathbb{E}_{\DT}\left\{\mathcal{L}\varphi_{n}\right\} \right| \leq  C \left( \DT^2 + \frac{1}{N\DT} \right) \|{\varphi}\|_{C^4}.
\end{equation}
It remains to estimate the term of order $\DT$ on the left hand side of the above inequality.
To this end, we apply the estimate \eqref{eqn:first-order-time-averaging}
to the function $\mathcal{L}\varphi$ and use the fact that $\mu(\mathcal{L}\varphi) = 0$, which implies
\[
\left| \frac{1}{N}\sum_{n=0}^{N-1}\mathbb{E}_{\DT}\left\{\mathcal{L}\varphi_{n}\right\} \right|
\leq C \| \mathcal{L}\varphi \|_{C^2}\left(\DT + \frac{1}{N\DT}\right)
\leq C' \| \varphi \|_{C^4}\left(\DT + \frac{1}{N\DT}\right)
\]
for some constant $C^{\prime}>0$.
The desired error estimate finally follows by combining the above estimate with~\eqref{eqn:second-order-time-averaging}.
\end{proof}

Let us emphasize that the above result does not rely on the ergodicity of the discrete chain $X_n$ (see~\cite{mattingly2010convergence}). However, as will be seen in Section~\ref{sec:CLR-scheme}, we need some ergodicity to study the CLR estimator for linear response estimation. We therefore discuss the ergodicity and the discrete Poisson equation associated with the discrete chain $X_n$ in the remainder of this section. 

\subsection{Ergodicity of the discrete chain}
The existence and uniqueness of an invariant probability measure 
of a Markov chain, and its exponential ergodicity, can be obtained 
by assuming that the evolution operator $P_{\DT}$ satisfies both a Lyapunov condition and a minorization condition \cite{hairer2011yet, meyn2012markov}.
Here, the Lyapunov condition is trivially satisfied since the configuration space $\mathcal{X}$ is compact (the Lyapunov function being the constant function equal to $1$).
As for the minorization condition, we need a slightly stronger version than the usual minorization condition, which requires that the constant and the probability measure are independent of the time step $h$ provided it is sufficiently small.   
\begin{assumption}[Uniform minorization condition]\label{assume:minorization}
Given the evolution operator~$P_{\DT}$ associated with either \eqref{eqn:first-order-scheme} or \eqref{eqn:second-order-scheme} and a fixed final integration time $T>0$,
there exist a maximum time step $\DT^* > 0$, a constant $\eta>0$ and a probability measure $\lambda$, such that for any $0 < \DT \leq \DT^*$ and any $x \in \mathcal{X}$,
\begin{equation}\label{eqn:minorization}
    P_{\DT}^{\lceil {T / \DT} \rceil}(x, dy) \geq \eta \lambda(dy).
\end{equation}
\end{assumption}
We emphasize that the constants $\eta$ and the probability measure $\lambda$ are independent of the time step $\DT$ provided that $\DT\leq \DT^*$.
This assumption can be justified for some important cases. See for example \cite{bou2013nonasymptotic, fathi2015error, FS17, lelievre2016partial} for discretizations of overdamped Langevin dynamics and \cite{leimkuhler2015computation,redon2016error} for discretizations of underdamped Langevin dynamics. The strategy of proofs of these works can be straightforwardly adapted to the schemes we consider here since the diffusion matrix $\sigma \sigma^\TR$ is bounded below (in the sense of symmetric matrices) by a positive constant.

We are now in position to state the exponential ergodicity result directly obtained from~\cite{hairer2011yet}, which also provides the existence and uniqueness of the invariant probability measure of the Markov chain. To state it, we introduce the space $B^\infty$ of bounded measurable functions, endowed with the norm $\|\varphi\|_{B^\infty} = \sup_{x \in \mathcal{X}} |\varphi(x)|$.

\begin{theorem}\label{thm:ergodicity}
There exists a maximum time step $\DT^* > 0$ such that, for any $\DT \in (0, \DT^*]$, the Markov chain associated with $P_{\DT}$ has a unique invariant measure $\mu_{\DT}$.
Furthermore, there exist constants $\kappa, C > 0$ that are independent of $\DT$ such that, for any function $\theta \in B^{\infty}$,
\begin{equation}
  \forall m \in \mathbb{N}_+, \qquad
  \left\| P_{\DT}^m \theta - \mu_{\DT}(\theta)\right \|_{B^{\infty}} \leq C  \mathrm{e}^{-\kappa m \DT} \|\theta\|_{B^{\infty}}.
\end{equation}
\end{theorem}

\subsection{The discrete time Poisson equation}
We present here some useful results on the Poisson equation associated with the discrete chain $X_n$:
\begin{equation}\label{eqn:discrete-Poisson-equation}
  \left[\frac{I - P_{\DT}}{\DT}\right] \widehat{\theta}_{\DT} = \theta - \mu_{\DT}(\theta).
\end{equation}
We first show that the solution $\widehat{\theta}_{\DT}$ is well defined. To this end, we introduce the Banach space $B_{\DT}^{\infty}$ of bounded measurable functions with average~0 with respect to~$\mu_\DT$. A direct consequence of Theorem~\ref{thm:ergodicity} is that, for any $m \in \mathbb{N}_+$,
\[
\|P_{\DT}^m \|_{\mathcal{B}(B_{\DT}^{\infty})} \leq C \mathrm{e}^{-\kappa m \DT},
\]
where $\| \cdot \|_{\mathcal{B}(B_{\DT}^{\infty})}$ is the operator norm on $B_{\DT}^{\infty}$. This estimates immediately implies that
\[
(I - P_{\DT})^{-1}
=
\sum_{m = 0}^{\infty} P_{\DT}^m
\]
is a convergent series and the inverse is well defined on $\mathcal{B}(B_{\DT}^{\infty})$, with 
\begin{equation}
  \label{eq:bound_inverse_Ph}
\left\|(I - P_{\DT})^{-1}\right\|_{\mathcal{B}(B_{\DT}^{\infty})}
\leq 
\sum_{m = 0}^{\infty} \|P_{\DT}^m\|_{\mathcal{B}(B_{\DT}^{\infty})} \leq \frac{C}{1-\mathrm{e}^{-\kappa \DT}}
\end{equation}
for $\DT \in (0, \DT^*]$ (with $\DT^*$ as defined in Theorem~\ref{thm:ergodicity}). Therefore, the solution to the discrete Poisson equation~\eqref{eqn:discrete-Poisson-equation} exists and is unique, with the following bound for any $\DT \in (0, \DT^*]$:
\begin{equation}\label{eqn:discrete-Poisson-solution-bounds}
\left\|\widehat{\theta}_{\DT} \right\|_{B^{\infty}} 
\leq 
\DT \left\|(I - P_{\DT})^{-1}\right\|_{\mathcal{B}(B_{\DT}^{\infty})}\|\theta - \mu_{\DT}(\theta)\|_{B_{\DT}^{\infty}}  \leq \frac{C \DT}{1-\mathrm{e}^{-\kappa \DT}}\|\theta - \mu_{\DT}(\theta)\|_{B_{\DT}^{\infty}}.
\end{equation}
Note in particular that the last term in the above series of inequalities is uniformly bounded for $\DT \in (0,\DT^*]$.

The above estimates provide a control of the discrete Poisson solution. However, this is not sufficient to justify the consistency of the CLR methods since this requires a control of the derivatives of $\widehat{\theta}_{\DT}$ as well. Unfortunately, the regularity of~$\widehat{\theta}_{\DT}$ is not so easy to obtain directly from the Poisson equation. It may even not hold for evolution operators which are not fully regularizing, as is the case for Metropolis-type evolutions. We can nevertheless overcome this difficulty by using the following technical result whose proof is postponed to Section~\ref{sec:approx-inverse-operator}. It shows that the solution to the discrete Poisson equation can be approximated at arbitrary order in powers of~$\DT$ by a smooth function.

\begin{theorem}\label{thm:approx-inverse}
Suppose that $P_{\DT}$ admits the following expansion in powers of~$\DT$: there exists $p \geq 1$ such that
\begin{equation}\label{eqn:semigroup-expansion}
P_{\DT} = I + \DT\mathcal{A}_1 + \ldots + \DT^{p+1}\mathcal{A}_{p+1} + \DT^{p+2} \mathcal{R}_{p, \DT},
\end{equation}
where $\mathcal{A}_1,\dots,\mathcal{A}_{p+1}$ are differential operators of finite order with smooth coefficients, and $\mathcal{R}_{p, \DT}$ is uniformly bounded for $h$~bounded in the following sense: For any $k \geq 1$, there exists $\ell_k \geq 1$, $R_k \in \mathbb{R}_+$ and $\DT_{k}^*$ such that 
\[
\forall \DT \in (0,\DT_{k}^*], \quad \forall \varphi \in \mathcal{S}, \qquad \left\| \mathcal{R}_{p, \DT}\varphi \right\|_{C^k} \leq R_k \|\varphi\|_{C^{\ell_k}}.
\]
Assume also that $\mathcal{A}_1^{-1}$ sends $\mathcal{S}_0$ to $\mathcal{S}_0$ in the following sense: For any $k \geq 1$, there exists $m_k \geq 1$, $K_k \in \mathbb{R}_+$ and $\DT_{k}^*$ such that  
\[
\forall \DT \in (0,\DT_{k}^*], \quad \forall \varphi \in \mathcal{S}_0, \qquad \left\| \mathcal{A}_1^{-1}\varphi \right\|_{C^k} \leq K_k \|\varphi\|_{C^{m_k}}.
\]
Then, for any $h>0$, there exists a function $\widetilde{\theta}_{\DT} \in \mathcal{S}_0$ which approximates the solution $\widehat{\theta}_{\DT}$ to the discrete Poisson in the following sense: 
\begin{equation}\label{eqn:approx-inverse-equality}
\Pi\left[\frac{I - P_{\DT}}{\DT}\right]\Pi \left(\widetilde{\theta}_{\DT} - \widehat{\theta}_{\DT}\right) = \DT^{p+1} \phi_{\DT, p, \theta},
\end{equation}
where $\phi_{\DT, p, \theta}\in  \mathcal{S}_0$ is uniformly bounded in the following sense: For any $r \in \mathbb{N}$, there exists $M_r \in \mathbb{R}_+$ and $\DT_r > 0$ such that
\begin{equation}
\label{eq:unif_bound_remainder_approx_inverse}
    \forall \DT \in (0,\DT_r], \qquad \|\phi_{\DT, p, \theta}\|_{C^r} \leq M_r.
\end{equation}
Moreover, when $\mathcal{A}_1 = \mathcal{L}$, there exists $C_\theta \in \mathbb{R}_+$ and $\DT^* > 0$ such that, for any $\DT \in (0,\DT^*]$,
\begin{equation}\label{eqn:diff-discrete-continuous-Poisson}
\left\|\widetilde{\theta}_{\DT} - \widehat{\theta}\right\|_{B^\infty} + \left\|\nabla \widetilde{\theta}_{\DT} - \nabla \widehat{\theta}\right\|_{B^\infty} \leq C_\theta \DT.
\end{equation}
\end{theorem}

Note that an immediate consequence of~\eqref{eqn:approx-inverse-equality} is that 
\begin{equation}\label{eqn:approx-inverse-equality-equivalent}
\widetilde{\theta}_{\DT} - \widehat{\theta}_{\DT} = h^{p+1}\left[\frac{I - P_{\DT}}{\DT}\right]^{-1}(\phi_{\DT, p, \theta} - \mu_{\DT}(\phi_{\DT, p, \theta})), 
\end{equation}
which allows to prove that $\left\|\widetilde{\theta}_{\DT} - \widehat{\theta}_{\DT}\right\|_{B^\infty} \leq C h^{p+1}$. Indeed, \eqref{eqn:approx-inverse-equality} is equivalent to
\[
\left[\frac{I - P_{\DT}}{\DT}\right](\widetilde{\theta}_{\DT} - \widehat{\theta}_{\DT}) = C_{\DT} + h^{p+1}\phi_{\DT, p, \theta},
\]
where the constant $C_h$ equals $\mu\left(\DT^{-1}[I - P_{\DT}](\widetilde{\theta}_{\DT} - \widehat{\theta}_{\DT})\right)$.
Since the left-hand side of the above equation is of zero mean with respect to $\mu_{\DT}$, it holds $C_{\DT} = -h^{p+1}\mu_{\DT}(\phi_{\DT, p, \theta})$. Finally, \eqref{eqn:approx-inverse-equality-equivalent} follows from~\eqref{eq:bound_inverse_Ph} since the inverse of~$\DT^{-1}(I - P_{\DT})$ can be applied to the function $\phi_{\DT, p, \theta} - \mu_{\DT}(\phi_{\DT, p, \theta})$.

\begin{remark}
Let us comment on the fact we have defined so far solutions to three Poisson equations: the solution $\widehat{\theta}$ to the Poisson equation~\eqref{eqn:continuous-Poisson-equation} associated with the continuous dynamics, the solution $\widehat{\theta}_{\DT}$ to the Poisson equation~\eqref{eqn:discrete-Poisson-equation} associated with the discrete dynamics, and the approximation $\widetilde{\theta}_{\DT}$ of~$\widehat{\theta}_{\DT}$ defined in $\eqref{eqn:approx-inverse-equality}$. Note that the actual Monte Carlo estimators in this paper do not involve $\widetilde{\theta}_{\DT}$, which are defined solely for the purpose of the mathematical analysis.  
\end{remark}


\section{Linear response estimation based on the  CLR scheme}\label{sec:CLR-scheme}

We are now in a position to present numerical schemes for linear response estimation based on the CLR method. 

\subsection{Weak first order CLR scheme}\label{sec:first-order-scheme}

We present in this section the weak first order CLR scheme that we propose for estimating the linear response index $\rho(\theta)$. Recall the continuous time CLR estimator defined in \eqref{eqn:continuous-CLR-estimator}. The weak first order CLR estimator we propose is 
\begin{equation}\label{eqn:first-order-CLR-estimator}
    \mathcal{M}_{\DT, N}^{[1]}(\theta) = \frac{1}{N}\sum_{n=0}^{N-1} (\theta_n - \mu_{\DT}(\theta)) Z_N,
\end{equation}
where 
\begin{equation}\label{eqn:discrete-Z}
Z_N = \sum_{n=0}^{N-1} (\sigma_n^{-1}F_n)^{\TR} \Delta W_n.
\end{equation}
Note that $\mathcal{M}_{\DT, N}^{[1]}(\theta)$ is simply a discrete approximation to \eqref{eqn:continuous-CLR-estimator}. It is important to note that we do not require the discrete time process $Z_n$ to be the likelihood ratio process associated with the discrete chain $X_n$. Instead, $Z_n$ is simply a discretization of the continuous time likelihood ratio process $Z(t)$. 

\begin{algorithm}
\caption{Pseudo-code of the first order CLR algorithm}
\label{alg:first-order-CLR}
\begin{algorithmic}[1]
\STATE{Choose integration time $T$, time step $\DT$, number of realizations~$s$}
\STATE{Define number of steps $N = \left \lfloor{T / \DT}\right \rfloor $}
\FOR{$i = 1 : s$}
\STATE{Initialize the starting state $X_0^{(i)} \sim \MUINI, Z_0^{(i)} = 0$ and running average $\alpha_0^{(i)} =0$}
\FOR{$n = 1 : N$}
\STATE{Update $\alpha_{n+1}^{(i)} = \alpha_n^{(i)} + N^{-1}\theta(X_n^{(i)})$}
\STATE{Generate random numbers $\Delta W_n^{(i)} \sim N(0, \DT \mathrm{Id}_d)$}
\STATE{Update $X_{n+1}^{(i)} = X_n^{(i)} + \Phi_{\DT}(X_n^{(i)}, \Delta W_n^{(i)})$}
\STATE{Update $Z_{n+1}^{(i)} = Z_n^{(i)} +  \sigma(X_n^{(i)})^{-1}F(X_n^{(i)}) \Delta W_n^{(i)} $}
\STATE{Increment $n$ as $n + 1$}
\ENDFOR
\ENDFOR
\STATE{Compute the empirical average $\displaystyle \bar{\alpha}_N = s^{-1} \sum_{i=1}^s \alpha_N^{(i)}$}
\RETURN{$\displaystyle s^{-1} \sum_{i = 1}^s \left(\alpha_N^{(i)} - \bar{\alpha}_N\right)Z_N^{(i)}$}
\end{algorithmic}
\end{algorithm}

The pseudo-code of the first order CLR algorithm is presented in Algorithm~ \ref{alg:first-order-CLR}. 
The CLR estimator is an ensemble average estimator based on multiple trajectories rather than an ergodic average estimator based on a single long trajectory. 

\begin{remark}
  Other centerings could be considered, in particular by finding the value $\alpha^*_{N,s}$ which minimizes the empirical variance of $s^{-1} \sum_{i = 1}^s (\alpha_N^{(i)} - a)Z_N^{(i)}$ with respect to~$a$. A simple computation shows that
  \[
   \alpha^*_{N,s} = \frac{\mathrm{Cov}_s(\alpha_N Z_N,Z_N)}{\mathrm{Cov}_s(Z_N,Z_N)}, 
 \]
 where
 \[
   \mathrm{Cov}_s(X,Y) = \frac1s \sum_{i=1}^s X^{(i)} Y^{(i)} - \left(\frac1s \sum_{i=1}^s X^{(i)}\right)\left(\frac1s \sum_{i=1}^s Y^{(i)}\right).
 \]
 Of course, $\alpha^*_{N,s}$ converges to $\mu_\DT(\theta)$ as $N,s \to +\infty$. Our numerical experience however shows that there is not much benefit from centering by $\alpha^*_{N,s}$ rather than by the empirical average~$\bar{\alpha}_N$, so that we therefore stick to the latter one for simplicity.
\end{remark}

\subsubsection{Consistency of the first order CLR scheme}

The following result shows that the estimator~\eqref{eqn:first-order-CLR-estimator} is consistent in the limits $\DT \to 0$ and $T = N\DT \to +\infty$.

\begin{theorem}\label{thm:first-order-error}
  Fix an observable~$\theta \in \mathcal{S}$ and consider the weak first order scheme~\eqref{eqn:first-order-scheme}. There exist $\DT^* > 0$ and $C \in \mathbb{R}_+$ such that, for any $\DT \in (0, \DT^*]$,
\begin{equation}\label{eqn:first-order-error}
   \left|\mathbb{E}_{\DT}\left\{\mathcal{M}_{\DT, N}^{[1]}(\theta)\right\} - \rho(\theta)\right| \leq C \left(\DT + \frac{1}{\sqrt{N\DT}}\right).
\end{equation}
\end{theorem}

Note that the bias has two origins: one part is related to the time step~$\DT$, and is proportional to~$\DT$, as expected for a scheme of weak order~1; the second part of the bias arises from the fact that the integration time~$T$ is finite, and scales as $T^{-1/2}$. The fact that the latter error is larger than the $1/T$ error for standard time averages (as studied in~\cite{mattingly2010convergence}) is due to the martingale~$Z_N$.

\begin{proof}
We rewrite the CLR estimator using the discrete Poisson equation \eqref{eqn:discrete-Poisson-equation}:
\[
\frac{1}{N}\sum_{n = 0}^{N-1}\mathbb{E}_{\DT}\left\{(\theta_n - \mu_{\DT}(\theta)) Z_N\right\} = \frac{1}{N\DT}\sum_{n = 0}^{N-1}\mathbb{E}_{\DT}\left\{(I - P_{\DT}) \widehat{\theta}_{\DT, n} Z_N\right\},
\]
where $P_\DT \widehat{\theta}_{\DT, n}$ stands for $(P_\DT \widehat{\theta}_{\DT})(X_n)$. Since $Z_N$ is of mean zero, it can be readily verified that the right-hand side of the above equation equals
\[
\frac{1}{N\DT}\sum_{n = 0}^{N-1}\mathbb{E}_{\DT}\left\{ \left[\Pi(I - P_{\DT})\Pi \widehat{\theta}_{\DT, n}\right] Z_N\right\}.
\]
The motivation for introducing the projection operator $\Pi$ into the estimator is to replace $\widehat{\theta}_{\DT}$ by its approximation $\widetilde{\theta}_{\DT}$ (see \eqref{eqn:approx-inverse-equality}) whose derivatives can be controlled. For the Euler-Maruyama scheme~\eqref{eqn:first-order-scheme}, it can be shown that the evolution semigroup admits the expansion
\begin{equation}\label{eqn:first-order-semigroup}
P_{\DT} \varphi = \varphi + \mathcal{L} \varphi \, \DT + \mathcal{A}_2 \varphi \, \DT^2 + \mathcal{R}_{1, \DT}\varphi \, \DT^3,
\end{equation}
where $\mathcal{A}_2, \mathcal{R}_{1, \DT}$ satisfy the assumptions of Theorem~\ref{thm:approx-inverse}; see for instance~\cite[Section~3.3.2]{lelievre2016partial} for explicit expressions, although the action of the operators $\mathcal{A}_2$ and $\mathcal{R}_{1, \DT}$ on the right-hand side of~\eqref{eqn:first-order-semigroup} need not be made precise. By choosing $p = 1$ in Theorem~\ref{thm:approx-inverse}, 
\[
\begin{aligned}
\frac{1}{N\DT}\sum_{n = 0}^{N-1}\mathbb{E}_{\DT}\left\{\left[\Pi(I - P_{\DT})\Pi \widehat{\theta}_{\DT, n}\right] Z_N\right\}
& =
    \frac{1}{N\DT}\sum_{n = 0}^{N-1}\mathbb{E}_{\DT}\left\{ \left[(I - P_{\DT}) \widetilde{\theta}_{\DT, n} \right]Z_N\right\}\\
    & \ \ - 
    \frac{\DT^2}{N}\sum_{n = 0}^{N-1}\mathbb{E}_{\DT}\left\{ \phi_{\DT, 1, \theta}(X_n) Z_N \right\},
\end{aligned}
\]
where the remainder term $\phi_{\DT, 1, \theta}$ is uniformly bounded in~$B^\infty$ for $\DT \in (0,\DT^*]$. Therefore, by reorganizing the sum for the first term on the right hand side, 
\begin{equation}\label{eqn:first-order-estimates-2}
\begin{split}
\left|\frac{1}{N\DT}\sum_{n = 0}^{N-1}\mathbb{E}_{\DT}\left\{\Pi(I - P_{\DT})\Pi \widehat{\theta}_{\DT, n} Z_N\right\} - \rho(\theta)\right|
& \leq 
\left|\frac{1}{N \DT}\sum_{n = 0}^{N-1}\mathbb{E}_{\DT}\left\{\left(\widetilde{\theta}_{\DT, n+1} - P_{\DT}\widetilde{\theta}_{\DT, n}\right) Z_N\right\} - \rho(\theta)\right| \\
& \ \ + \left|\frac{1}{N \DT}\mathbb{E}_{\DT}\left\{\left(\widetilde{\theta}_{\DT, N} - \widetilde{\theta}_{\DT, 0} \right)Z_N\right\}\right|\\
&\ \ + \left| \frac{\DT^2}{N}\sum_{n = 0}^{N-1}\mathbb{E}_{\DT}\left\{ \phi_{\DT, 1, \theta}(X_n) Z_N \right\}\right|.
\end{split}
\end{equation}
Let us estimate the various terms on the right-hand side of~\eqref{eqn:first-order-estimates-2}. The last term of the right-hand side of~\eqref{eqn:first-order-estimates-2} can be bounded by Lemma~\ref{lemma:elementary-coarse-estimate} in Section~\ref{sec:estimate_elmentary} as 
\begin{equation}\label{eqn:first-order-estimate-RHS-3}
\left|\frac{\DT^2}{N}\sum_{n = 0}^{N-1}\mathbb{E}_{\DT}\left\{ \phi_{\DT, 1, \theta}(X_n) Z_N \right\}\right| \leq C\DT^{3/2}.
\end{equation} 
As for the second term of the right-hand side of \eqref{eqn:first-order-estimates-2}, a simple application of the Cauchy--Schwarz inequality gives (using that $\widetilde{\theta}_{\DT}$ is uniformly bounded by~\eqref{eqn:diff-discrete-continuous-Poisson}),
\begin{equation}\label{eqn:first-order-estimate-RHS-2}
\left|\frac{1}{N \DT}\mathbb{E}_{\DT}\left\{\left(\widetilde{\theta}_{\DT, N} - \widetilde{\theta}_{\DT, 0} \right)Z_N\right\}\right| \leq \frac{C}{Nh} \sqrt{\mathbb{E}_{\DT}(Z_N^2)}  \leq \frac{C}{\sqrt{N \DT}}.
\end{equation}
Hence, it remains to estimate the first term on the right hand side of \eqref{eqn:first-order-estimates-2}.
By a simple conditioning argument on the increments of the discrete martingale~$Z_N$, and noting that $\widetilde{\theta}_{\DT, n+1} - P_{\DT}\widetilde{\theta}_{\DT, n}$ are discrete martingale increments,
\[
\frac{1}{N \DT}\sum_{n = 0}^{N-1}\mathbb{E}_{\DT}\left\{\left(\widetilde{\theta}_{\DT, n+1} - P_{\DT}\widetilde{\theta}_{\DT, n}\right) Z_N\right\}
=
\frac{1}{N \DT}\sum_{n = 0}^{N-1}\mathbb{E}_{\DT}\left\{\left(\widetilde{\theta}_{\DT, n+1} - P_{\DT}\widetilde{\theta}_{\DT, n}\right)(Z_{n+1} - Z_n)\right\}.
\]
We next expand $\widetilde{\theta}_{\DT, n+1}$ in terms of the increment $\Phi_{\DT, n}$, i.e., 
\begin{equation*}
\begin{split}
\widetilde{\theta}_{\DT, n+1} = 
\widetilde{\theta}_{\DT, n}
&+ \widetilde{\theta}_{\DT, n}^{\prime}  \Phi_{\DT, n}
+ \frac{1}{2}\widetilde{\theta}_{\DT, n}^{(2)}  \Phi_{\DT, n}^{2} 
+ \frac{1}{6}\widetilde{\theta}_{\DT, n}^{(3)}  \Phi_{\DT, n}^{3} + r_{\DT, \theta, n},
\end{split}
\end{equation*} 
where the remainder term reads
\[
r_{\DT, \theta, n}=
\left(\frac{1}{6}\int_0^1 u^3  \widetilde{\theta}_{\DT}^{(4)}(X_n + u \Phi_{\DT, n})  \,du\right)\Phi_{\DT, n}^{4}.
\]
Plugging in $\Phi_{\DT,n} = b_n \DT + \sigma_n \Delta W_n$ for the Euler-Maruyama scheme~\eqref{eqn:first-order-scheme} and rearranging terms leads to
\begin{equation}\label{eqn:first-order-Taylor-expansion}
\begin{split}
\widetilde{\theta}_{\DT, n+1} 
= \widetilde{\theta}_{\DT, n} 
+ \widetilde{\theta}_{\DT, n}^{\prime} \sigma_n \Delta W_n
+ \left\{\widetilde{\theta}_{\DT, n}^{\prime} b_n \DT +  \frac{1}{2}\widetilde{\theta}_{\DT, n}^{(2)} \sigma_n^2 \Delta W_n^2\right\}\\
+ \left\{ \widetilde{\theta}_{\DT, n}^{(2)} b_n \sigma_n \DT\Delta W_n
+ \frac{1}{6}\widetilde{\theta}_{\DT, n}^{(3)} \sigma_n^3 \Delta W_n^3  \right\}
+ \psi_{\DT, \theta, n},
\end{split}
\end{equation}
where the remainder term is of order~$h^2$:
\[
\psi_{\DT, \theta, n} = \frac{1}{2}\left(\widetilde{\theta}_{\DT, n}^{(2)} b_n^2 \DT^2 + \widetilde{\theta}_{\DT, n}^{(3)}b_n \sigma_n^2 \DT \Delta W_n^2\right)
+ \frac{1}{2}\widetilde{\theta}_{\DT, n}^{(3)}b_n^2 \sigma_n \DT^2 \Delta W_n 
+ \frac{1}{6}\widetilde{\theta}_{\DT, n}^{(3)}b_n^3\DT^3 + r_{\DT, 
\theta, n}.
\]
Gathering the expansions \eqref{eqn:first-order-semigroup} and \eqref{eqn:first-order-Taylor-expansion}, a simple calculation shows that
\begin{equation}
  \label{eq:expansion_discrete_martingale_increment}
  \begin{aligned}
  \widetilde{\theta}_{\DT, n+1} - P_{\DT}\widetilde{\theta}_{\DT, n} = \widetilde{\theta}_{\DT, n}^{\prime} \sigma_n \Delta W_n
& + \left\{\widetilde{\theta}_{\DT, n}^{\prime} b_n \DT +  \frac{1}{2}\widetilde{\theta}_{\DT, n}^{(2)} \sigma_n^2 \Delta W_n^2 - h \mathcal{L}\widetilde{\theta}_{\DT, n}\right\}\\
& + \left\{ \widetilde{\theta}_{\DT, n}^{(2)} b_n \sigma_n \DT\Delta W_n + \frac{1}{6}\widetilde{\theta}_{\DT, n}^{(3)} \sigma_n^3 \Delta W_n^3  \right\}
+ \widetilde{\psi}_{\DT, \theta, n},
\end{aligned}
\end{equation}
where the remainder term $\widetilde{\psi}_{\DT, \theta, n}$ is of order~$h^2$, so that, recalling $Z_{n+1}-Z_n = \sigma_n^{-1} F_n \Delta W_n$,
\begin{equation*}
\begin{aligned}
& \frac{1}{\DT}\mathbb{E}_{\DT}\left\{\left(\widetilde{\theta}_{\DT, n+1} - P_{\DT}\widetilde{\theta}_{\DT, n}\right)(Z_{n+1} - Z_n) \right\}\\
& \qquad =
\mathbb{E}_{\DT}\left\{\widetilde{\theta}_{\DT, n}^{\prime} F_n \right\}
+ \mathbb{E}_{\DT}\left\{ \widetilde{\theta}_{\DT, n}^{(2)} b_n F_n
+ \frac{1}{2}\widetilde{\theta}_{\DT, n}^{(3)} \sigma_n^2 F_n  \right\}\DT
+ \widetilde{\Psi}_{\DT, \theta, n} \DT^2
\end{aligned}
\end{equation*}
for some $\widetilde{\Psi}_{\DT, \theta, n}$ uniformly bounded in~$B^\infty$ for $\DT$ sufficiently small.
Summing the above equality over $n$ and then bounding the $\DT$ and $\DT^2$ terms,
\[
\left|\frac{1}{N \DT}\sum_{n = 0}^{N-1}\mathbb{E}_{\DT}\left\{\left(\widetilde{\theta}_{\DT, n+1} - P_{\DT}\widetilde{\theta}_{\DT, n}\right)(Z_{n+1} - Z_n)\right\} - \rho(\theta)\right|
\leq
\left|\frac{1}{N}\sum_{n = 0}^{N-1}\mathbb{E}_{\DT}\left\{\widetilde{\theta}_{\DT, n}^{\prime} F_n \right\} -\rho(\theta)\right| + C\DT
\]
for some constant $C>0$. Since $\widetilde{\theta}_{\DT}^{\prime} F \in \mathcal{S}$ and $\|\widetilde{\theta}_{\DT}^{\prime} F\|_{B^{\infty}}$ is uniformly bounded for sufficiently small $\DT$ (see~\eqref{eqn:diff-discrete-continuous-Poisson}), Proposition~\ref{thm:time-averaging-consistency} shows that
\[
\left|\frac{1}{N}\sum_{n = 0}^{N-1}\mathbb{E}_{\DT}\left\{\widetilde{\theta}_{\DT, n}^{\prime} F_n \right\} - \mu\left(\widetilde{\theta}_{\DT}^{\prime} F\right)\right|
\leq 
C \left(\DT + \frac{1}{N\DT}\right)
\]
for some constant $C\in \mathbb{R}_+$ that is independent of $\DT$. Since $\left|\mu(\widetilde{\theta}_{\DT}^{\prime} F) - \mu(\widehat{\theta}^{\prime} F)\right| \leq C \DT$
by Theorem~\ref{thm:approx-inverse}, and $\mu(\widehat{\theta}^{\prime} F) = \rho(\theta)$ by Proposition~\ref{thm:alternative-linear-response-formula}, we immediately have 
\begin{equation}\label{eqn:first-order-estimate-RHS-1}
\left|\frac{1}{N \DT}\sum_{n = 0}^{N-1}\mathbb{E}_{\DT}\left\{\left(\widetilde{\theta}_{\DT, n+1} - P_{\DT}\widetilde{\theta}_{\DT, n}\right)Z_N\right\}
 -  \rho(\theta)\right| \leq C\left(\DT + \frac{1}{N\DT}\right).
\end{equation}
Finally, 
combining the estimates \eqref{eqn:first-order-estimate-RHS-3}, \eqref{eqn:first-order-estimate-RHS-2} and \eqref{eqn:first-order-estimate-RHS-1} 
leads to the result.
\end{proof}

\subsubsection{Variance analysis of the first order CLR scheme}

The following result shows that the variance of the estimator~\eqref{eqn:first-order-CLR-estimator} is bounded uniformly with respect to the integration time for sufficiently small time steps.

\begin{theorem}\label{thm:first-order-var}
  Fix an observable $\theta \in \mathcal{S}$ and consider the weak first order scheme~\eqref{eqn:first-order-scheme}. There exist $\DT^* > 0$ and $C_1,C_2 \in \mathbb{R}_+$ such that, for any $\DT \in (0, \DT^*]$,
\begin{equation}\label{first-order-var-estim}
  \mathrm{Var}_{\DT}\left\{\mathcal{M}_{\DT, N}^{[1]}(\theta)\right\} \leq C_1 + C_2\left(\DT + \frac{1}{T}\right).
\end{equation}
\end{theorem}

In essence, the dominant part of the variance~$C_1$ is due to the asymptotic variance of the CLR estimator associated with the underlying continuous dynamics, as given by Theorem~\ref{thm:LR-variance}. The extra term $C_2(h + T^{-1})$ comes from the discretization error and the finiteness of the integration time.

\begin{remark}
 The results of Theorems~\ref{thm:first-order-error} and~\ref{thm:first-order-var} provide a guide to choosing the parameters for the simulation by equilibriating the various sources of errors. More precisely, the bias is the sum of a term of order $1/(Nh)$ and a term of order $h^\alpha$ (with $\alpha=1$ for the first order scheme, but we will see below in Theorem~\ref{thm:second-order-error} that a second order accuracy $\alpha=2$ can be achieved), while the statistical error scales at dominant order as $s^{-1/2}$ when $s$ realizations are considered (as in Algorithm~\ref{alg:first-order-CLR}). The computational cost scales on the other hand as $Ns$, where $N$ is the number of iterations to reach the integration time $T=Nh$. Therefore, the optimization of the parameters amounts to minimizing a function of the form
 \[
 a h^\alpha + \frac{b}{Nh} + \frac{c}{\sqrt{s}}, \qquad Ns = K,
 \]
 with the computational cost $K$ fixed. The Euler-Lagrange equations with respect to $h,N$ show that
 \[
N h^{\alpha+1} = \frac{b}{a \alpha}, \qquad N^{3/2}h = \frac{2b\sqrt{K}}{c},
 \]
 which allows to choose the values of $N,s,h$ as a function of~$K$ provided estimates of $a,b,c$ are available.
\end{remark}

\begin{proof}
  We bound the second moment of the estimator $\mathcal{M}_{\DT, N}^{[1]}(\theta)$. Fist, using the discrete Poisson equation~\eqref{eqn:discrete-Poisson-equation} and the equality
  \[
    \sum_{n=0}^{N-1} (I - P_{\DT}) \widehat{\theta}_{\DT, n} = \widehat{\theta}_{\DT, 0} - \widehat{\theta}_{\DT, N} + \sum_{n=0}^{N-1} \widehat{\theta}_{\DT, n+1}- P_\DT \widehat{\theta}_{\DT,n},
  \]
  we immediately have
\begin{equation*}
    \begin{split}
      &\mathbb{E}_{\DT}\left\{ \left(\mathcal{M}_{\DT, N}^{[1]}(\theta)\right)^2\right\}=
    \mathbb{E}_{\DT}\left\{\left( \left[\frac{1}{N\DT} \sum_{n=0}^{N-1}(I - P_{\DT}) \widehat{\theta}_{\DT, n}\right] Z_N \right)^2\right\}\\
& \quad \leq
\frac{2}{N^2\DT^2} \mathbb{E}_{\DT}\left\{\left(\sum_{n=0}^{N-1}\left(\widehat{\theta}_{\DT, n+1} - P_{\DT}\widehat{\theta}_{\DT, n}\right) Z_N\right)^2  \right\}
+
\frac{2}{N^2\DT^2}\mathbb{E}_{\DT}\left\{ \left((\widehat{\theta}_{\DT, N} - \widehat{\theta}_{\DT, 0})Z_N \right)^2\right\}.
    \end{split}
\end{equation*}
Since $\widehat{\theta}_{\DT, N} -\widehat{\theta}_{\DT, 0}$ is uniformly bounded in $\DT$ and $\mathbb{E}_{\DT}(Z_N^2)\leq CN\DT$, it is easy to verify that the second term on the right-hand side of the above inequality is bounded by $C T^{-1}$. Hence, it only remains to estimate the first term on the right-hand side of the above inequality. For each $n = 0, 1, \ldots, N-1$, it is convenient to denote the martingale increments by 
\[
\xi_n = \widehat{\theta}_{\DT, n+1} - P_{\DT}\widehat{\theta}_{\DT, n}, \qquad \eta_n = Z_{n+1} - Z_n.
\]
We can then write
\begin{equation*}
\begin{split}
\mathbb{E}_{\DT}\left\{\left(\sum_{n=0}^{N-1}\left(\widehat{\theta}_{\DT, n+1} - P_{\DT}\widehat{\theta}_{\DT, n}\right) Z_N\right)^2  \right\}
=
\sum_{n_1,n_2, n_3, n_4=0}^{N-1}\mathbb{E}_{\DT}\left\{  \xi_{n_1} \xi_{n_2} \eta_{n_3}\eta_{n_4}\right\}.
\end{split}
\end{equation*}
Note that both $\xi_n$ and $\eta_n$ depend on $\Delta W_n$.
Since the sequence of normal random variables $\Delta W_n$ are independently and identically distributed, we can verify that
\begin{equation}\label{eqn:first-order-scheme-var-proof}
\sum_{n_1,n_2, n_3, n_4=0}^{N-1} \!\!\!\!\!\!\!\! \mathbb{E}_{\DT}\left\{  \xi_{n_1} \xi_{n_2} \eta_{n_3}\eta_{n_4}\right\} 
= 
\sum_{n_1 = 0}^{N-1}\sum_{n_2=0}^{N-1} \mathbb{E}_{\DT}\{\xi_{n_1}^2\eta_{n_2}^2\}
+ 
2 \!\!\!\!\!\!\!\!\! \sum_{0 \leq n_1 < n_2 \leq N-1}^{N-1} \!\!\!\!\!\!\!\!\! \mathbb{E}_{\DT}\{\xi_{n_1}\xi_{n_2}\eta_{n_1}\eta_{n_2}\}.
\end{equation}
In order to estimate the two terms on the right-hand side of the above equation, we need to expand $\xi_n$ in terms of $\DT$. However, recall that $\widehat{\theta}_{\DT}$ is not necessarily in $\mathcal{S}$. Again, the strategy is to replace $\widehat{\theta}_{\DT}$ by $\widetilde{\theta}_{\DT}$ using the approximate inverse argument. To this end, we consider the trivial decomposition
\[
  \xi_n = \widetilde{\xi}_n + (\xi_n - \widetilde{\xi}_n),
  \qquad
  \widetilde{\xi}_n = \widetilde{\theta}_{\DT,n+1} - P_{\DT}\widetilde{\theta}_{\DT, n}.
\]
A crude estimate for $\xi_n - \widetilde{\xi}_n$ is obtained with~\eqref{eqn:approx-inverse-equality-equivalent}. Indeed, 
\[
\xi_n - \widetilde{\xi}_n  = -\DT^3\left[I-P_{\DT}\right]^{-1}  (\phi_{\DT, 1, \theta}(X_{n+1}) - \mu_{\DT}(\phi_{\DT, 1, \theta})) + \DT^3\left[I-P_{\DT}\right]^{-1}P_{\DT} (\phi_{\DT, 1, \theta}(X_n) - \mu_{\DT}(\phi_{\DT, 1, \theta})) .
\]
Since $\phi_{\DT, 1, \theta}$ is uniformly bounded in~$B^\infty$ and $\|\left[I-P_{\DT}\right]^{-1}\|_{\mathcal{B}(B_{\DT}^{\infty})} \leq 2\kappa^{-1} \DT^{-1}C$ for $h$ sufficiently small by~\eqref{eq:bound_inverse_Ph}, we immediately have 
\begin{equation}\label{eqn:first-order-scheme-var-diff}
\forall n \geq 0, \qquad |\xi_n - \widetilde{\xi}_n| \leq C \DT^{2} \qquad \mathrm{a.s.}
\end{equation}
Let us next state some moments estimates involving $\widetilde{\xi}_n$ and $\eta_n$:
\[
  \begin{aligned}
\mathbb{E}_{\DT}\left\{\widetilde{\xi}_n^2\eta_n^2\right\} & \leq \mathbb{E}_{\DT}\left\{\left(\widetilde{\theta}_{\DT, n}^{\prime} F_n\right)^2\right\} \DT^2 + C\DT^3, \\
\mathbb{E}_{\DT}\left\{\widetilde{\xi}_{n_1}^2 \eta_{n_2}^2\right\} & \leq \mathbb{E}_{\DT}\left\{\left(\widetilde{\theta}_{\DT, n_1}^{\prime} \sigma_{n_1} U_{n_2}\right)^2\right\} \DT^2 + C\DT^3, \quad n_1 \neq n_2,  \\
\mathbb{E}_{\DT}\left\{\widetilde{\xi}_{n_1}\widetilde{\xi}_{n_2}\eta_{n_1}\eta_{n_2} \right\} & \leq \mathbb{E}_{\DT}\left\{\widetilde{\theta}_{\DT, n_1}^{\prime} F_{n_1} \widetilde{\theta}_{\DT, n_2}^{\prime} F_{n_2}\right\}\DT^2 + C \DT^3,
\quad n_1 \neq n_2,
\end{aligned}
\]
which can be easily derived from the expansion~\eqref{eq:expansion_discrete_martingale_increment}.
Note that all the coefficients of the $\DT^2$ terms in the above estimates are uniformly bounded in both $\DT$ and $n$ since we assume that the state space is compact. 
Combining the above moment estimates with \eqref{eqn:first-order-scheme-var-proof} and \eqref{eqn:first-order-scheme-var-diff}, a simple estimation shows that there exist positive constants $C_1,C_2$ such that, for $\DT$ sufficiently small, 
\[
\left|\frac{1}{N^2\DT^2}\sum_{n_1,n_2, n_3, n_4=0}^{N-1}\mathbb{E}_{\DT}\left\{  \xi_{n_1} \xi_{n_2} \eta_{n_3}\eta_{n_4}\right\} \right| \leq C_1 + C_2\DT.
\]
The final estimate follows by taking the error $CT^{-1}$ into account.
\end{proof}


\subsection{Weak second order CLR scheme}\label{sec:second-order-scheme}
{
In this section, we study a weak second order CLR scheme based on the numerical discretization scheme~\eqref{eqn:second-order-scheme}.
We show that in the scalar case, $d=1$,
an appropriate modification of the original weight process $Z_n$ is 
\begin{equation}
  \label{eqn:discrete-Z_2-scalar}
  Y_{n + 1} = Y_n + (\sigma_n^{-1} F_n)\left( \Delta W_n +  \frac{1}{2}\left(-\mathcal{L}\sigma_n + \mathcal{K} b_n\right) \sigma_n^{-1} \Delta W_n \DT\right)\,.
\end{equation}
In the multi-dimensional case
we are only able to treat the case when the diffusion $\sigma(x)$ is state independent, i.e., $\sigma_n = \sigma$ for all $n$ and a given constant matrix $\sigma$.
In this case of an additive noise the term $\mathcal{L}\sigma$ in \eqref{eqn:discrete-Z_2-scalar} vanishes and an appropriate choice of the modified weight process is 
\begin{equation}
  \label{eqn:discrete-Z_2-vector}
  Y_{n + 1} = Y_n + (\sigma^{-1} F_n)^{\TR}\left( \Delta W_n +  \frac{1}{2}\left(\mathcal{K} b_n\right)^{\TR} \sigma^{-\TR} \Delta W_n \DT\right),
\end{equation}
where $\mathcal{K} b$ is a matrix with columns $\mathcal{K}^k b$ for $k = 1, \ldots, d$. In order to keep the notation simple we present the calculations for the scalar case and refer to Appendix~\ref{sec:appendix} for details about algebraic derivations in the multi-dimensional case with an additive noise.
}
It is easy to verify that the modified process $Y_n$ is still a zero-mean martingale. With the definition \eqref{eqn:discrete-Z_2-vector}, the second order CLR estimator is
\begin{equation}\label{eqn:second-order-CLR-estimator}
    \mathcal{M}_{\DT, N}^{[2]}(\theta) = \frac{1}{N} \left[\sum_{n=0}^{N-1} \theta_n - \mu_{\DT}(\theta)\right] Y_N + \frac{\DT}{2N}\sum_{n=0}^{N-1}\nabla\theta_n^{\TR}  F_n,
\end{equation}
where the extra term at the end is a correction term specific for the second-order scheme.
We present the weak second order CLR algorithm in Algorithm~\ref{alg:second-order-CLR}. 
\begin{algorithm}
\caption{Pseudo-code for the weak second order CLR algorithm}
\label{alg:second-order-CLR}
\begin{algorithmic}[1]
\STATE{Choose integration time $T$, time step $\DT$, number of realizations~$s$}
\STATE{Define number of steps $N = \left \lfloor{T / \DT}\right \rfloor $}
\FOR{$i = 1 : s$}
\STATE{Initialize the starting state $X_0^{(i)} \sim \MUINI, Y_0^{(i)} = 0$ and running averages $\alpha_0^{(i)} =0, \beta_0^{(i)} = 0$}
\FOR{$n = 1 : N$}
\STATE{Update $\alpha_{n+1}^{(i)} = \alpha_n^{(i)} + N^{-1}\theta(X_n^{(i)})$}
\STATE{Update $\beta_{n+1}^{(i)} = \beta_n^{(i)} + N^{-1} \nabla \theta(X_N^{(i)})^{\TR} F(X_N^{(i)})$}
\STATE{Generate random numbers $\Delta W_n^{(i)} \sim \mathcal{N}(0, \DT I_d)$} and $V_n^{(i)}$ according to~\eqref{eqn:V-dist}
\STATE{Update $X_{n+1}^{(i)} = X_n^{(i)} + \Phi_{\DT}(X_n^{(i)},\Delta w_n^{(i)},V_n^{(i)})$}
\STATE{Update $\displaystyle Y_{n+1}^{(i)}$ according to \eqref{eqn:discrete-Z_2-scalar} in the scalar setting or \eqref{eqn:discrete-Z_2-vector}} in the multi-dimensional setting
\STATE{Increment $n$ as $n + 1$}
\ENDFOR
\ENDFOR
\STATE{Compute the empirical average $\displaystyle \bar{\alpha}_N = s^{-1} \sum_{i=1}^s \alpha_N^{(i)}$}
\RETURN{$\displaystyle s^{-1} \sum_{i = 1}^s \left(\alpha_N^{(i)} - \bar{\alpha}_N\right)Y_N^{(i)} + \frac{\DT}{2s} \sum_{i=1}^s\beta_N^{(i)}$}
\end{algorithmic}
\end{algorithm}

\subsubsection{Consistency of the second order CLR scheme}

The following result is the counterpart of Theorem~\ref{thm:first-order-error} for the weak second order CLR scheme~\eqref{eqn:second-order-scheme}.

\begin{theorem}\label{thm:second-order-error}
  Consider an observable $\theta \in \mathcal{S}$ and the weak second order scheme~\eqref{eqn:second-order-scheme}. There exist $\DT^* > 0$ and $C \in \mathbb{R}_+$ such that, for any $\DT \in (0, \DT^*]$,
\begin{equation}\label{eqn:second-order-error}
   \left|\mathbb{E}_{\DT}\left( \mathcal{M}_{\DT, N}^{[2]}(\theta) \right) - \rho(\theta)\right| \leq C \left(\DT^2 + \frac{1}{\sqrt{N\DT}}\right).
\end{equation}
\end{theorem}

The proof provided below assumes the modified martingale~\eqref{eqn:discrete-Z_2-scalar} and only works for the scalar setting. 
For the multi-dimensional setting, the modified martingale~\eqref{eqn:discrete-Z_2-vector} leads to a second order CLR scheme under the additional assumption of additive noise.
Nevertheless, the result is already relevant as such for applications such as molecular dynamics where the noise is often considered to be additive. For completeness, we provide the algebra for justifying the correctness of~\eqref{eqn:discrete-Z_2-vector} in Appendix~\ref{sec:appendix}.

\begin{proof}
  The strategy of the proof is the same as for the weak first order scheme. We write it as for the other proofs in the scalar case ($d = 1$). The second order discretization scheme then reads
  \begin{equation*}
    \begin{split}
      X_{n+1}  = X_n + & \sigma_n \Delta W_n + b_n \DT 
      + \frac{1}{2}\mathcal{K}\sigma_n \left((\Delta W_n)^2  - \DT\right) + \frac{1}{2}\left(\mathcal{K}b_n + \mathcal{L}\sigma_n\right) \Delta W_n \DT + \frac{1}{2}\mathcal{L}b_n\DT^2,
\end{split}
\end{equation*}
together with the modified weight 
\begin{equation}
  \label{eq:modified_weight_1D}
  Y_{n+1} = Y_n + \sigma_n^{-1} F_n \Delta W_n + \frac{1}{2}\sigma_n^{-2} F_n\left(-\mathcal{L}\sigma_n + \mathcal{K}b_n\right) \Delta W_n\DT.
\end{equation}
We can split the estimator into three separate terms:
\begin{equation}\label{eqn:second-order-estimate-proof}
\begin{split}
&\left|\frac{1}{N}\sum_{n = 0}^{N-1}\mathbb{E}_{\DT}\left\{(\theta_n - \mu_{\DT}(\theta)) Y_N\right\} + \frac{\DT}{2N}\sum_{n=0}^{N-1} \mathbb{E}_{\DT}\{\theta_n^{\prime}F_n\} - \rho(\theta)\right|\\
& \qquad \leq 
 \left|\frac{1}{N \DT}\sum_{n = 0}^{N-1}\mathbb{E}_{\DT}\left\{\left(\widetilde{\theta}_{\DT, n+1} - P_{\DT}\widetilde{\theta}_{\DT, n}\right) Y_N\right\}  + \frac{\DT}{2N}\sum_{n=0}^{N-1} \mathbb{E}_{\DT}\{\theta_n^{\prime}F_n\} - \rho(\theta) \right|\\
& \ \ \qquad + \left|\frac{1}{N \DT}\mathbb{E}_{\DT}\left\{\left(\widetilde{\theta}_{\DT, N} - \widetilde{\theta}_{\DT, 0} \right) Y_N\right\}\right|
+ \left|\frac{\DT^3}{N} \sum_{n = 0}^{N-1}\mathbb{E}_{\DT}\left\{ \phi_{\DT, 2, \theta}(X_n)  Y_N \right\}\right|,
\end{split}
\end{equation}
where $\phi_{\DT, 2, \theta}$ is the function obtained from Theorem~\ref{thm:approx-inverse} with $p = 2$. Using the same argument as for the proof of Theorem~\ref{thm:first-order-error}, the second and last terms on the right-hand side of \eqref{eqn:second-order-estimate-proof} can be bounded by $C N^{-1/2} \DT^{-1/2}$ and $C\DT^{5/2}$, respectively.

It remains to estimate the first term on the right-hand side of~\eqref{eqn:second-order-estimate-proof}. First, denoting by $\Phi_{\DT,n} = X_{n+1}-X_n = \Phi_{\DT}(X_n,\Delta W_n,V_n)$, we expand 
\begin{equation*}
\begin{split}
\widetilde{\theta}_{\DT, n+1} = 
\widetilde{\theta}_{\DT, n}
&+ \widetilde{\theta}_{\DT, n}^{\prime}  \Phi_{\DT, n}
+ \frac{1}{2}\widetilde{\theta}_{\DT, n}^{(2)}  \Phi_{\DT, n}^{2} 
+ \frac{1}{6}\widetilde{\theta}_{\DT, n}^{(3)}  \Phi_{\DT, n}^{3} + r_{\DT, \theta, n}, 
\end{split}
\end{equation*}
with the remainder
\[
r_{\DT, \theta, n} = \left(\frac{1}{6}\int_0^1 u^3  \widetilde{\theta}_{\DT}^{(4)}(X_n + u \Phi_{\DT, n})  \,du\right)\Phi_{\DT, n}^{4}.
\]
Gathering the terms with the same powers of~$\DT$,
\begin{equation*}
  \begin{split}
&  \widetilde{\theta}_{\DT, n+1} 
= \widetilde{\theta}_{\DT, n} 
 + \widetilde{\theta}_{\DT, n}^{\prime}\sigma_n \Delta W_n + \left\{ \widetilde{\theta}_{\DT, n}^{\prime}\left(b_n \DT + \frac{1}{2}\sigma_n \sigma_n^{\prime} ((\Delta W_n)^2 - \DT)\right) + \frac{1}{2}\widetilde{\theta}_{\DT, n}^{(2)}\sigma_n^2 (\Delta W_n)^2  \right\}
\\
& + \left\{ \frac{1}{2} \widetilde{\theta}_{\DT, n}^{\prime}(\mathcal{K}b_n + \mathcal{L} \sigma_n) \Delta W_n \DT 
 +
 \widetilde{\theta}_{\DT, n}^{(2)} \sigma_n \left(b_n \DT + \frac{1}{2}\mathcal{K}\sigma_n ((\Delta W_n)^2 - \DT)\right)\Delta W_n + 
 \frac{1}{6} \widetilde{\theta}_{\DT, n}^{(3)}\sigma_n^3 (\Delta W_n)^3
 \right\} + \psi_{\DT, \theta, n},
\end{split}
\end{equation*}
where $\psi_{\DT, \theta, n}$ is a remainder term of order~$h^2$ in the following sense: for all $k \geq 1$, there exists $C_k \in \mathbb{R}_+$ such that $\mathbb{E}_\DT\left(\left|\psi_{\DT, \theta, n}\right|^k\right) \leq C_k\DT^{2k}$. Recall that the second order scheme \eqref{eqn:second-order-scheme} admits the expansion
\begin{equation*}
P_{\DT} \widetilde{\theta}_{\DT, n}= \widetilde{\theta}_{\DT, n} 
+ \mathcal{L} \widetilde{\theta}_{\DT, n}\DT
+ \frac12 \mathcal{L}^2 \widetilde{\theta}_{\DT, n}\DT^2 
+ \mathcal{A}_3 \widetilde{\theta}_{\DT, n}\DT^3 
+ \mathcal{R}_{2, \DT}\widetilde{\theta}_{\DT, n}\DT^4
\end{equation*}
for some operators $\mathcal{A}_3, \mathcal{R}_{2, \DT}$. Multiplying $\widetilde{\theta}_{\DT, n+1} - P_{\DT} \widetilde{\theta}_{\DT, n}$ by $Y_{n+1} - Y_n$, a simple calculation leads to
\begin{equation*}
\begin{split}
& \frac{1}{\DT} \mathbb{E}_{\DT}\left\{\left(\widetilde{\theta}_{\DT, n+1} - P_{\DT}\widetilde{\theta}_{\DT, n}\right) Y_N\right\}  = \frac{1}{\DT}\mathbb{E}_{\DT}\left\{\left(\widetilde{\theta}_{\DT, n+1} - P_{\DT}\widetilde{\theta}_{\DT, n}\right)(Y_{n+1} - Y_n) \right\}\\
& \qquad =
\mathbb{E}_{\DT}\left\{\widetilde{\theta}_{\DT, n}^{\prime} F_n \right\}
+ \mathbb{E}_{\DT}\left\{\left( \widetilde{\theta}_{\DT, n}^{\prime} \sigma_n^{-1}\mathcal{K}b_n + \widetilde{\theta}_{\DT, n}^{(2)} (b_n + \mathcal{K}\sigma_n) +\frac{1}{2} \widetilde{\theta}_{\DT, n}^{(3)}\sigma_n^2\right)F_n \right\}\DT
+ \widetilde{\Psi}_{\DT, \theta, n} \DT^2,
\end{split}
\end{equation*}
where $\widetilde{\Psi}_{\DT, \theta, n}$ is uniformly bounded for $\DT \in (0, \DT^*]$. Now, the key observation is that 
\[
\widetilde{\theta}_{\DT}^{\prime} \sigma^{-1}\mathcal{K}b + \widetilde{\theta}_{\DT}^{(2)} (b + \mathcal{K}\sigma) +\frac{1}{2} \widetilde{\theta}_{\DT}^{(3)}\sigma^2 = \left(\mathcal{L}\widetilde{\theta}_{\DT}\right)^{\prime}.
\]
Hence, 
\begin{equation*}
\frac{1}{\DT}\mathbb{E}_{\DT}\left\{\left(\widetilde{\theta}_{\DT, n+1} - P_{\DT}\widetilde{\theta}_{\DT, n}\right)(Y_{n+1} - Y_n) \right\} = \mathbb{E}_{\DT}\left\{\widetilde{\theta}_{\DT, n}^{\prime} F_n \right\}
+ \mathbb{E}_{\DT}\left\{\left( \mathcal{L}\widetilde{\theta}_{\DT, n} \right)^{\prime}F_n \right\}\DT
+ \widetilde{\Psi}_{\DT, \theta, n} \DT^2.
\end{equation*}
Taking the average of the above equation over $n$ and collecting the remainder terms, we obtain the following estimate for the first term of the right-hand side of~\eqref{eqn:second-order-estimate-proof}:
\begin{equation*}
\begin{split}
&\left|\frac{1}{N\DT}\sum_{n=0}^{N-1}\mathbb{E}_{\DT}\left\{\left(\widetilde{\theta}_{\DT, n+1} - P_{\DT}\widetilde{\theta}_{\DT, n}\right)Y_N \right\} + \frac{\DT}{2N}\sum_{n=0}^{N-1} \mathbb{E}_{\DT}\{\theta_n^{\prime}F_n\} - \rho(\theta) \right|
\\
& \qquad \leq  \left|\frac{1}{N}\sum_{n=0}^{N-1}\mathbb{E}_{\DT}\left\{\widetilde{\theta}_{\DT, n}^{\prime} F_n \right\} - \rho(\theta) + \frac{\DT}{N}\sum_{n=0}^{N-1}\mathbb{E}_{\DT}\left\{\left(\mathcal{L}\widetilde{\theta}_{\DT, n}\right)^{\prime}F_n  + \frac{1}{2}\theta_n^{\prime}F_n \right\} \right| + C\DT^2
\end{split}
\end{equation*}
for some constant $C>0$.
Note that here there exists, by Proposition~\ref{thm:time-averaging-consistency} and Theorem~\ref{thm:approx-inverse}, a constant $C \in \mathbb{R}_+$ independent of $\DT$ such that
\[
  \begin{aligned}
    \left|\frac{1}{N}\sum_{n=0}^{N-1}\mathbb{E}_{\DT}\left\{\widetilde{\theta}_{\DT, n}^{\prime} F_n \right\} - \mu\left(\widetilde{\theta}_{\DT}^{\prime} F\right)\right| & \leq C\left(\DT^2 + \frac{1}{N\DT}\right), \\
    \left|\frac{1}{N}\sum_{n=0}^{N-1}\mathbb{E}_{\DT}\left\{ \left(\mathcal{L}\widetilde{\theta}_{\DT, n}\right)^{\prime} F_n + \frac{1}{2}\theta_n^{\prime}F_n\right\} - \mu\left( \left(\mathcal{L}\widetilde{\theta}_{\DT}\right)^{\prime} F + \frac{1}{2}\theta^{\prime}F\right)\right| & \leq C\left(\DT^2 + \frac{1}{N\DT}\right).
    \end{aligned}
\]
Therefore, 
\begin{equation}\label{eqn:second-order-error-before-correction}
\begin{split}
&\left|\frac{1}{N\DT}\sum_{n=0}^{N-1}\mathbb{E}_{\DT}\left\{\left(\widetilde{\theta}_{\DT, n+1} - P_{\DT}\widetilde{\theta}_{\DT, n}\right)Y_N \right\}
+
\frac{\DT}{2N}\sum_{n=0}^{N-1} \mathbb{E}_{\DT}\{\theta_n^{\prime}F_n\} - \rho(\theta) \right|
\\
& \qquad \leq
\left|\mu\left(\widetilde{\theta}_{\DT}^{\prime}F\right)
-
\rho(\theta)
+
\mu\left( (\mathcal{L}\widetilde{\theta}_{\DT})^{\prime} F + \frac{1}{2}\theta^{\prime}F\right)\DT
\right|
+ C\left(\DT^2 +\frac{1}{N\DT}\right)\\
& \qquad \leq
\left|\mu(\widetilde{\theta}_{\DT}^{\prime}F)
-
\rho(\theta) - \frac{1}{2}\mu(\theta^{\prime}F)\DT\right|
+
\left|
\mu\left((\mathcal{L}\widetilde{\theta}_{\DT})^{\prime} F - (\mathcal{L}\widehat{\theta})^{\prime}F\right)\DT
\right|
+ C\left(\DT^2 +\frac{1}{N\DT}\right),
\end{split}
\end{equation}
where we have used the continuous time Poisson equation~\eqref{eqn:continuous-Poisson-equation} for the last inequality.
Next, we show that $\widetilde{\theta}_{\DT}$ on the right-hand side of the above inequality can be replaced by $\widehat{\theta}$ with a controllable error. To this end, we set $p=2$ in~\eqref{eqn:diff-discrete-continuous-Poisson-pointwise} and note that $\mathcal{A}_1 = \mathcal{L}$ and $\mathcal{A}_2 = \mathcal{L}^2/2$ for the second order scheme \eqref{eqn:second-order-scheme}, so that 
\begin{equation}\label{eqn:diff-Poisson-solution}
\widetilde{\theta}_{\DT} - \widehat{\theta} = \frac12 (\theta - \mu(\theta)) \DT -  \left(\widetilde{\mathcal{A}}_1^{-1} \widetilde{\mathcal{A}}_2 \widetilde{\mathcal{A}}_1^{-1}\widetilde{\mathcal{A}}_2 \widetilde{\mathcal{A}}_1^{-1} - \widetilde{\mathcal{A}}_1^{-1}\widetilde{\mathcal{A}}_3 \widetilde{\mathcal{A}}_1^{-1}\right)(\theta - \mu(\theta))\DT^2.
\end{equation}
Since $\mu(\widehat{\theta}^{\prime}F)=\rho(\theta)$ by Theorem~\ref{thm:alternative-linear-response-formula}, we can easily verify that 
\[
\left| \mu\left(\widetilde{\theta}_{\DT}^{\prime}F\right)  - \rho(\theta) - \frac{1}{2} \mu(\theta^{\prime} F) \DT \right| \leq C\DT^2
\]
for some constant $C$.
Similarly, we can also deduce from \eqref{eqn:diff-Poisson-solution} that
\[
\left|\mu\left( (\mathcal{L}\widetilde{\theta}_{\DT})^{\prime} F
- (\mathcal{L}\widehat{\theta})^{\prime} F\right)\right| \leq C \DT
\]
for some constant $C$. Substituting the above two estimates into the right-hand side of  \eqref{eqn:second-order-error-before-correction} gives
\begin{equation}\label{eqn:second-order-error-after-correction}
\begin{split}
&\left|\frac{1}{N\DT}\sum_{n=0}^{N-1}\mathbb{E}_{\DT}\left\{\left(\widetilde{\theta}_{\DT, n+1} - P_{\DT}\widetilde{\theta}_{\DT, n}\right)Y_N \right\}
+
\frac{\DT}{2N}\sum_{n=0}^{N-1} \mathbb{E}_{\DT}\{\theta_n^{\prime}F_n\} - \rho(\theta) \right| 
\leq C\left(\DT^2 +\frac{1}{N\DT}\right).
\end{split}
\end{equation}
 The error estimate \eqref{eqn:second-order-error} now follows by combining the above estimate with the estimates for the second and third terms on the right-hand side of $\eqref{eqn:second-order-estimate-proof}$.
\end{proof}

\subsubsection{Variance analysis of second order CLR scheme}

The following result is the counterpart of Theorem~\ref{thm:first-order-var} for the weak second order scheme~\eqref{eqn:second-order-scheme}.

\begin{theorem}\label{thm:second-order-var}
Consider an observable $\theta \in \mathcal{S}$ and the weak second order scheme~\eqref{eqn:second-order-scheme}. There exist $\DT^* > 0$ and $C_1,C_2 \in \mathbb{R}_+$ such that, for any $\DT \in (0, \DT^*]$,
\begin{equation}\label{second-order-var-estim}
  \mathrm{Var}_{\DT}\left\{\mathcal{M}_{\DT, N}^{[2]}(\theta)\right\} \leq C_1 + C_2\left(\DT + \frac{1}{N\DT}\right).
\end{equation}
\end{theorem}

\begin{proof}
  We write as usual the proof in the scalar case, in which case the weak second order estimator reads
\[
\mathcal{M}_{\DT, N}^{[2]}(\theta) = \frac{1}{N}\left[\sum_{n=0}^{N-1} \theta_n - \mu_{\DT}(\theta) \right] Y_N + \frac{\DT}{2N}\sum_{n=0}^{N-1}\theta_n^{\prime}  F_n,
\]
with $Y_N$ defined in~\eqref{eq:modified_weight_1D}. We bound the variance of the estimator by two separate parts,
\begin{equation*}
\mathrm{Var}_{\DT}\left\{\mathcal{M}_{\DT, N}^{[2]}(\theta)\right\} \leq 2\mathrm{Var}_{\DT}\left\{\frac{1}{N}\sum_{n=0}^{N-1} (\theta_n - \mu_{\DT}(\theta))Y_N\right\} + 2 \mathrm{Var}_{\DT}\left\{\frac{\DT}{2N}\sum_{n=0}^{N-1}\theta_n^{\prime} F_n\right\}.
\end{equation*}
The estimate of the first term of the right-hand side of the above inequality is similar to that of the first order estimator and hence can be shown to be bounded by $C_1 + C_2(\DT + (N\DT)^{-1})$ for some constants $C_1, C_2 >0$. The second term can be directly bounded by $Ch^2$ since $\theta$ and~$F$ are bounded.
\end{proof}

\subsection{General weak second order CLR scheme}\label{sec:generel-2nd-order}
The proof of Theorem~\ref{thm:second-order-error} in fact suggests a general 
strategy for constructing a second order CLR estimator on top of an arbitrarily given second order discretization scheme.
The key point is to remove all the $\mathcal{O}(\DT)$ errors from the one step increment
$\DT^{-1}\mathbb{E}_{\DT}\{(\widetilde{\theta}_{\DT}(X_{n+1}) - P_{\DT}\widetilde{\theta}_{\DT}(X_n)) (Y_{n+1} - Y_n)\}$. 
We present the strategy in the one-dimensional case for dynamics with multiplicative noise.

Suppose that a given weak second order discretization scheme satisfies the recursive formula
\[
X_{n+1} = X_n + c_0(X_n; G_n)\DT^{1/2} + c_1(X_n; G_n) \DT
+ c_2(X_n; G_n) \DT^{3/2} + R_\DT(X_n; G_n) \DT^2,
\]
where the coefficients $c_i$ depend on $X_n$, the random vectors $G_n$ are used to generate the increments $\Delta W_n$ and $R_\DT$ is some remainder term. This can be straightforwardly generalized to account for a dependence on additional random numbers, as in~\eqref{eqn:second-order-scheme} or Metropolis-type schemes. Note also that $G_n$ are not necessarily Gaussian, as long as they satisfy some moment conditions, as made precise below. We require that 
\[
  c_0(X_n; G_n) = \sigma_n G_n,
\]
\begin{equation}
  \label{eq:moment_conditions_general}
  \mathbb{E}(c_1(X_n;G_n)G_n) = 0, \qquad \mathbb{E}\left(c_0^2(X_n;G_n)G_n\right) = 0,
\end{equation}
and that $c_1,c_2,R_\DT$ are uniformly bounded in the sense that, for any $k \geq 1$, there exists $C_k \in \mathbb{R}_+$ and $\DT_k^* > 0$ such that $\mathbb{E}(|R_\DT(X_n;G_n)|^k) \leq C_k$ for any $0 < \DT \leq \DT_k^*$ (and similar estimates for $c_1,c_2$). Following the proof of Theorem~\ref{thm:second-order-error}, we expand $\widetilde{\theta}_{\DT, n+1} - P_{\DT}\widetilde{\theta}_{\DT, n}$ in powers of~$\DT$:
\begin{equation*}
    \begin{split}
&      \widetilde{\theta}_{\DT, n+1} - P_{\DT}\widetilde{\theta}_{\DT, n} \\
      & =
\widetilde{\theta}_{\DT, n}^{\prime} c_0(X_n; G_n) \DT^{1/2} + 
\left(\widetilde{\theta}_{\DT, n}^{\prime} c_1(X_n; G_n) + \frac{1}{2}\widetilde{\theta}_{\DT, n}^{(2)} c_0^2(X_n; G_n) - \mathcal{L}\widetilde{\theta}_{\DT, n}\right)\DT \\
& \ \ + 
\left(\widetilde{\theta}_{\DT, n}^{\prime}c_2(X_n; G_n) + \widetilde{\theta}_{\DT, n}^{(2)}c_0(X_n; G_n) c_1(X_n; G_n) + \frac{1}{6}\widetilde{\theta}_{\DT, n}^{(3)} c_0^3(X_n; G_n)\right)  \DT^{3/2} +
\mathcal{O}(\DT^2).
\end{split}
\end{equation*}
Next, we consider a modification of the weight process of the form
\[
Y_{n+1} = Y_n + \sigma_n^{-1} F_n \left(G_n + \gamma(X_n; G_n)\DT\right)\DT^{1/2},
\]
with $\gamma(X_n; G_n)$ to be determined. Hence, the expansion of
$\DT^{-1}\mathbb{E}_{\DT}\{(\widetilde{\theta}_{\DT}(X_{n+1}) - P_{\DT}\widetilde{\theta}_{\DT}(X_n)) (Y_{n+1} - Y_n)\}$ in powers of~$\DT$ reads:
\begin{equation}\label{eqn:general-expansion}
\begin{split}
\mathbb{E}_{\DT}\left\{\widetilde{\theta}_{\DT, n}^{\prime}c_0(X_n; G_n)\sigma_n^{-1}F_n G_n\right\}
+
\mathbb{E}_{\DT}\left\{\left(\widetilde{\theta}_{\DT, n}^{\prime}c_0(X_n; G_n) \gamma(X_n; G_n) + \widetilde{\theta}_{\DT, n}^{\prime} c_2(X_n; G_n)G_n\right.\right.\\
\left.\left.+ \widetilde{\theta}_{\DT, n}^{(2)}c_0(X_n; G_n)c_1(X_n; G_n)G_n + \frac{1}{6}\widetilde{\theta}_{\DT, n}^{(3)}c_0^3(X_n; G_n)G_n\right)\sigma_n^{-1}F_n\right\} \DT + \mathcal{O}(\DT^2),
\end{split}
\end{equation}
where we used~\eqref{eq:moment_conditions_general} to eliminate the terms of order~$\DT^{1/2}$ in the above expansion. Similar conditions guarantee that the terms of order~$\DT^{3/2}$ vanish.

In order to achieve a second order accuracy for linear response, we need to remove the $\mathcal{O}(\DT)$ errors from~\eqref{eqn:general-expansion}. As shown in the proof of Theorem~\ref{thm:second-order-error}, both the first and second terms of~\eqref{eqn:general-expansion} contain terms of order~$\DT$. Indeed, since we explicitly assume that $c_0(X_n; G_n) = \sigma_n G_n$, the first term of \eqref{eqn:general-expansion} becomes $\mathbb{E}_{\DT}\{\widetilde{\theta}_{\DT, n}^{\prime}F_n\}$, so that, by~\eqref{eqn:diff-discrete-continuous-Poisson-pointwise}, 
\[
\mathbb{E}_{\DT}\{\widetilde{\theta}_{\DT, n}^{\prime}F_n\}
=
\mathbb{E}_{\DT}\{\widehat{\theta}_n^{\prime}F_n\} + \frac{1}{2}\mathbb{E}_{\DT}\{\theta_n^{\prime}F_n\}\DT + \mathcal{O}(\DT^2),
\]
where the $\mathcal{O}(\DT)$ error can be removed by a correction a posteriori (as provided by the second term of~\eqref{eqn:second-order-CLR-estimator}). We next choose an appropriate correction $\gamma(X_n; G_n)$ for the second term of~\eqref{eqn:general-expansion} to vanish at dominant order in~$\DT$. However, such a correction function~$\gamma(X_n; G_n)$ may involve the solution to the discrete Poisson solution $\widehat{\theta}_{\DT}$ or its approximation $\widetilde{\theta}_{\DT}$, which would make it impossible to compute the modified weight process in practice. A more practical alternative is to look for functions~$d_1(X_n)$ and $d_2(X_n)$ such that the term of order~$\DT$ in~\eqref{eqn:general-expansion} is equal to   
\begin{equation}\label{eqn:general-expansion-dt-term}
\mathbb{E}_{\DT}\left\{d_1(X_n) \left(\mathcal{L}\widetilde{\theta}_{\DT, n}\right)^{\prime}
+ d_2(X_n) \mathcal{L}\widetilde{\theta}_{\DT, n}\right\}\DT.
\end{equation}
The approximate discrete Poisson solution $\widetilde{\theta}_{\DT}(X_n)$ can then be replaced at dominant order in~$\DT$ by the solution of the continuous time Poisson equation in view of~\eqref{eqn:diff-discrete-continuous-Poisson-pointwise}. Comparing the above formula with the second term of \eqref{eqn:general-expansion} and matching the terms that involve the same
order of derivatives of $\widetilde{\theta}_{\DT}$, we end up with the following system of equations:
\begin{equation}\label{eqn:general-expansion-system}
    \begin{split}
        &\mathbb{E}_{\DT}\left\{\frac{1}{6}\widetilde{\theta}_{\DT, n}^{(3)}c_0^3(X_n; G_n)\sigma_n^{-1} F_n G_n\right\}
        =
        \mathbb{E}_{\DT}\left\{\frac{1}{2}\widetilde{\theta}_{\DT, n}^{(3)}d_1(X_n)\sigma_n^2\right\},\\
        &\mathbb{E}_{\DT}\left\{\widetilde{\theta}_{\DT, n}^{(2)}c_0(X_n; G_n)c_1(X_n; G_n)\sigma_n^{-1}F_n G_n\right\}
        =
        \mathbb{E}_{\DT}\left\{\widetilde{\theta}_{\DT, n}^{(2)}\left(d_1(X_n)\left(b_n + \sigma_n\sigma_n^{\prime}\right) + \frac{1}{2}d_2(X_n)\sigma_n^2\right)\right\},\\
        &\mathbb{E}_{\DT}\left\{\widetilde{\theta}_{\DT, n}^{\prime}\left(c_0(X_n; G_n)\gamma(X_n; G_n) + c_2(X_n; G_n)G_n\right)\sigma_n^{-1}F_n \right\}
        =
        \mathbb{E}_{\DT}\left\{\widetilde{\theta}_{\DT, n}^{\prime}\left( d_2(X_n)b_n + d_1(X_n)b_n^{\prime}\right)\right\}.
    \end{split}
\end{equation}
Note that $d_1(X_n)$ can be identified from the first equation, then $d_2(X_n)$ from the second, and finally $\gamma(X_n; G_n)$ from the third one. Let us mention that these factors are independent of $\widetilde{\theta}_{\DT}$ and hence are computable. More precisely, the first equality holds for

\begin{equation}
  \label{eq:d_1}
d_1(X_n) = \frac13 \mathbb{E}_{G_n}\left[c_0^3(X_n; G_n)\sigma_n^{-1}F_n G_n\right],
\end{equation}
the second for
\begin{equation}
  \label{eq:d_2}
d_2(X_n) = 2\sigma_n^{-2} \left( \mathbb{E}_{G_n}\left[c_0(X_n ;G_n)c_1(X_n; G_n)\sigma_n^{-1}F_n G_n\right] - (b_n+\sigma_n\sigma_n^{\prime})d_1(X_n) \right),
\end{equation}
so that $\gamma$ is found by solving
\begin{equation}
  \label{eq:gamma}
  \begin{aligned}
    F_n \mathbb{E}_{G_n}\left[\gamma(X_n; G_n) G_n\right] = d_2(X_n) b_n & + d_1(X_n) b_n^{\prime} 
    - \sigma_n^{-1} F_n \mathbb{E}_{G_n}\left[c_2(X_n; G_n)G_n\right].
  \end{aligned}
\end{equation}

It remains to rewrite \eqref{eqn:general-expansion-dt-term} in a computable form. To this end, we use the estimate \eqref{eqn:diff-discrete-continuous-Poisson-pointwise} and 
the continuous time Poisson equation to rewrite \eqref{eqn:general-expansion} as
\begin{equation*}
\begin{aligned}
&\mathbb{E}_{\DT}\left\{d_1(X_n) \left(\mathcal{L}\widehat{\theta}_n\right)^{\prime} + d_2(X_n) \mathcal{L}\widehat{\theta}_n \right\}\DT + \mathcal{O}(\DT^2)\\
& \qquad = -\mathbb{E}\left\{d_1(X_n) \theta_n^{\prime} + d_2(X_n) \left(\theta_n - \mu(\theta)\right) \right\}\DT + \mathcal{O}(\DT^2).
\end{aligned}
\end{equation*}
The dominant contribution of order~$\DT$ can be corrected a posteriori (as done in~\eqref{eqn:second-order-CLR-estimator}) since it does not involve the solution to the discrete Poisson equation or its approximation.

Let us follow the above strategy to recover the modified weight process of Theorem~\ref{thm:second-order-error}. 
The coefficients~$c_i$ for the second order discretization~\eqref{eqn:second-order-scheme} are
\[
c_0(X_n; G_n) = \sigma_n G_n, \ \  c_1(X_n; G_n) = b_n + \frac{1}{2}\mathcal{K}\sigma_n (G_n^2 - 1) , \ \ c_2(X_n; G_n) = \frac{1}{2}(\mathcal{K}b_n + \mathcal{L}\sigma_n) G_n.
\]
Plugging them into~\eqref{eq:d_1} and~\eqref{eq:d_2} leads to $d_1 = F$ and $d_2 = 0$. A possible solution for~\eqref{eq:gamma} is then 
\[
\gamma(X_n; G_n) = \frac{1}{2}\sigma_n^{-1}(\mathcal{K}b_n - \mathcal{L}\sigma_n) G_n, 
\] 
which allows to recover~\eqref{eqn:discrete-Z_2-scalar}.

\section{Computational benchmark}
\label{sec:numerical-results}
We present an example demonstrating that the derived estimates are sharp 
and reliable for the weak first and second order schemes described in the previous sections. At the same time this example also indicates that for some observables the first order scheme may be sufficiently accurate and it can computationally outperform the second order scheme for a certain range of time steps $\DT$. 
The benchmark example is defined on the periodic domain $\mathbb{T} = \mathbb{R} \backslash \mathbb{Z}$ for the gradient dynamics defined by the potential $V(x) = \frac{1}{2} \cos(2\pi x)$, i.e., the drift function $b(x) = -V^\prime(x)$, with the additive noise 
$\sigma(x) = \sqrt{2}$, hence
$$
dX(t) = \pi \sin(2\pi X(t)) \, dt + \sqrt{2} \, dW(t)\,.
$$ 
We have chosen the observable as $\theta (x) = b(x)$ and the external forcing $F(x) = 1$ (which is indeed not the gradient of a smooth periodic function).

Estimating the bias of the estimators~\eqref{eqn:first-order-CLR-estimator} and~\eqref{eqn:second-order-CLR-estimator} with respect to the time step $\DT$ is computationally expensive as it requires independent sampling over long trajectories in order to achieve a good approximation to the stationary distribution and to control the variance of the estimator. In our simulations we used the time horizon $T=10^2$ for equilibration and $s=5\times 10^7$ independent samples in Algorithms~\ref{alg:first-order-CLR} and~\ref{alg:second-order-CLR} in order for the statistical error to be sufficiently small. The $95\%$ confidence intervals, while plotted in Figure~\ref{fig:sensitivity}, are at the limit of the figure resolution.
Estimated values of the sensitivity index $\rho(\theta)$ are depicted in Figure~\ref{fig:sensitivity}. The importance of a properly corrected second order estimator \eqref{eqn:second-order-CLR-estimator} is demonstrated by including computed values from the estimator without corrections.

\begin{figure}
    \centering
    \includegraphics[width=0.8\textwidth]{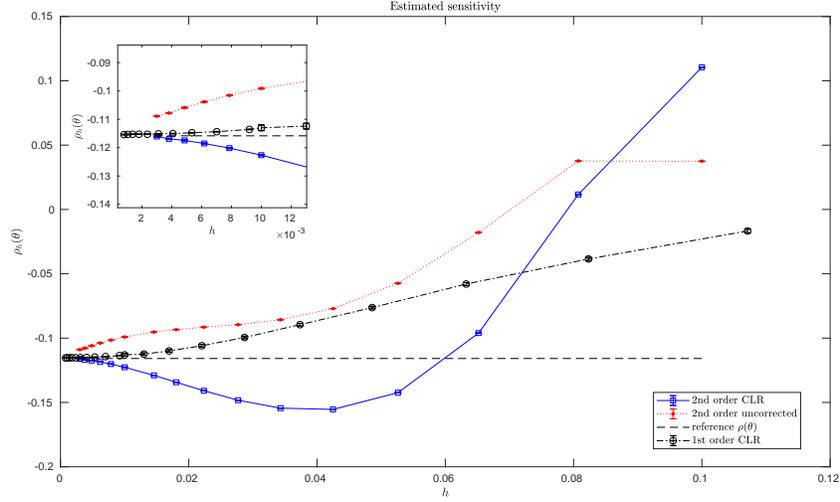}
    \caption{The sensitivity $\rho(\theta)$ estimated for different values of the time step $\DT$. The estimates are obtained from the first order scheme (marked by $\circ$) and the second order scheme (marked by $\square$). The estimates marked by $\bullet$ are obtained from the second order estimator without the correction term. The inset depicts a detail for a range of smaller time steps $h$. The reference value $\rho(\theta)$ has been computed by solving Fokker-Planck equation using numerical quadratures.}
    \label{fig:sensitivity}
\end{figure}

The convergence rates are estimated from error values obtained at the beginning of the asymptotic regime in~$\DT$. The observed convergence rate for the first order scheme was estimated as  $1.40\pm 0.06$ and for the second order scheme as $1.80\pm 0.05$. The error convergence is depicted in Figure~\ref{fig:error_estimators}. The error convergence plot for the CLR estimator also clearly demonstrates the necessity of the correcting term (see \eqref{eqn:second-order-CLR-estimator}) for the second order CLR sensitivity estimator. 

\begin{figure}
    \centering
    \includegraphics[width=0.6\textwidth]{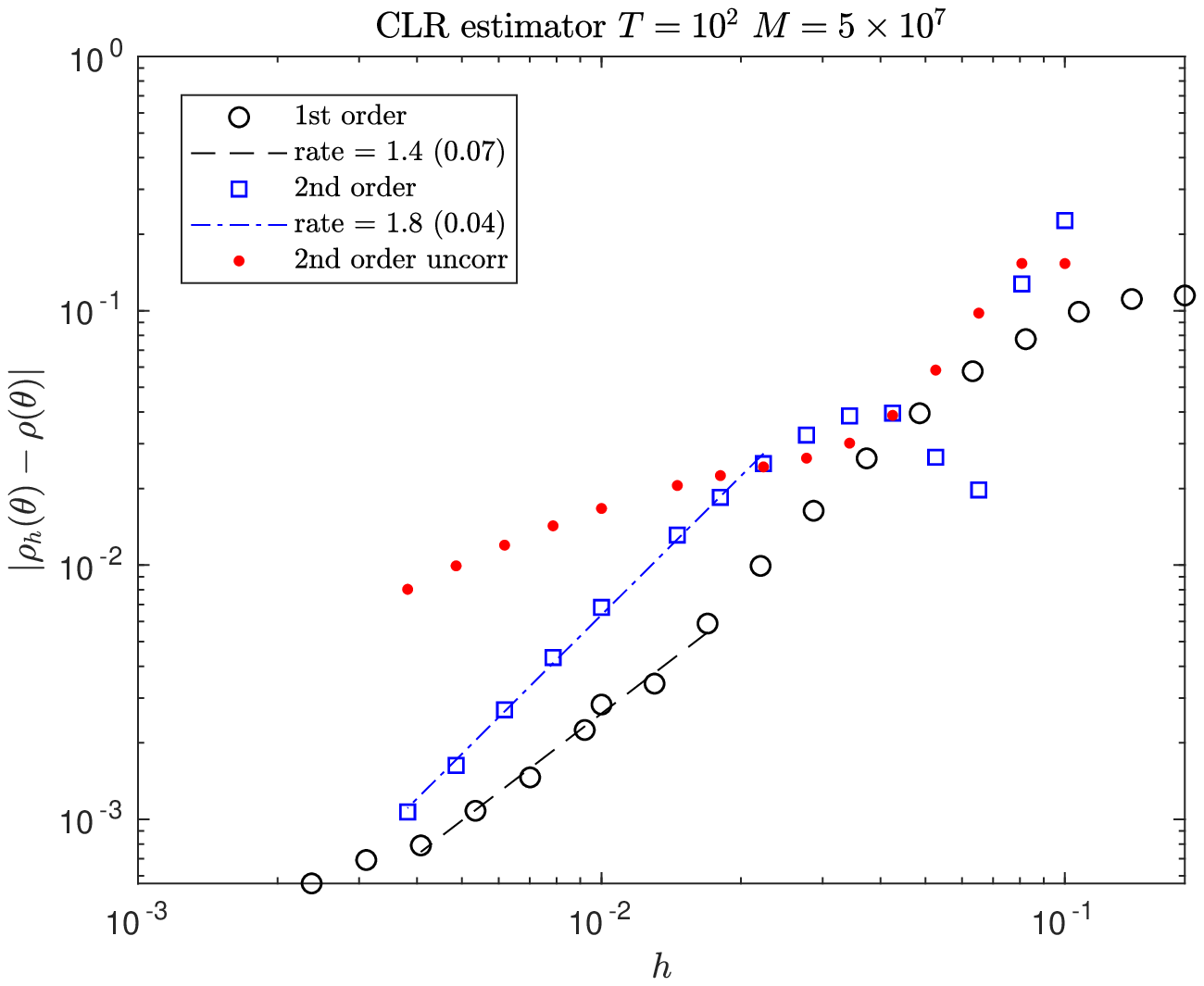}
    \caption{Convergence of the sensitivity estimators: the error for the estimated $\rho(\theta)$ obtained for different values of the time step $\DT$ (log-log scale). The estimates are obtained from the first order scheme (marked by $\circ$) and the second order scheme (marked by $\square$).  The error estimates marked by $\bullet$ are obtained from the second order estimator without the correction term. The reference value $\rho(\theta)$ has been computed by solving Fokker-Planck equation using numerical quadratures.  }
    \label{fig:error_estimators}
\end{figure}

An important feature of the proposed sensitivity estimator is the variance behavior of the CLR estimator $\mathcal{M}^{[k]}_{\DT,N}$ as stated in Theorem~\ref{thm:first-order-var} for $k=1$ and in 
Theorem~\ref{thm:second-order-var} for $k=2$. As the time horizon $T\equiv \DT N$ tends to infinity the variance is bounded by a constant. This result is demonstrated in Figure~\ref{fig:varT} which depicts,
for the fixed timesteps $\DT=10^{-2}$ and $\DT=10^{-3}$, the convergence of the estimated sensitivity $\rho_h(\theta)$ as well as
the estimated variance $\mathrm{Var}_{\DT}[\mathcal{M}^{[2]}_{\DT,N}]$ of the CLR estimator when increasing the time horizon $T\equiv \DT N$. The estimates in both cases 
($\mathbb{E}_{\DT}[\mathcal{M}^{[2]}_{h,N}]$, $\mathrm{Var}_{\DT}[\mathcal{M}^{[2]}_{h,N}]$) are obtained by averaging over $s=10^6$ independent sample trajectories of the physical time~$T$.  The first order estimator ($k=1$) exhibits a similar behaviour.

\begin{figure}
    \centering
    \includegraphics[width=0.6\textwidth]{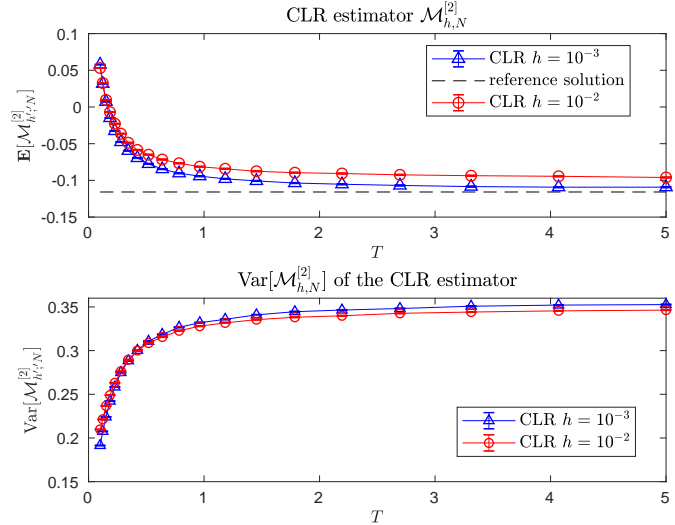}
    \caption{The mean and variance of the second order CLR estimator 
    $\mathcal{M}^{[2]}_{h,N}$ for the finite time-horizon trajectories $T\equiv\DT N$ estimated for each time horizon from $s=10^6$ independent samples.}
    \label{fig:varT}
\end{figure}


\section{Technical results}\label{sec:proofs}

We provide in this section two technical results to facilitate the proofs of the results presented in this paper. 

\subsection{Approximate inverse operator}
\label{sec:approx-inverse-operator}

We present here the proof of Theorem~\ref{thm:approx-inverse}, which gives error bounds for an approximate solution to the discrete Poisson equation~\eqref{eqn:discrete-Poisson-equation}. We follow the construction of the approximate inverse operator discussed in~\cite{leimkuhler2015computation, lelievre2016partial}. As in these works, we provide in fact an explicit construction of the approximate solution, whose derivatives we can indeed control. Recall that the interest of working with an approximate solution is that the solution~$\widehat{\theta}$ to the Poisson equation defined via the operator $\DT^{-1}[I - P_{\DT}]$ is well defined on $B^\infty_\DT$. However, we do not have control of its derivatives. This is however of paramount importance for us to establish the convergence result of the numerical schemes through Taylor-like expansions. 

We recall that we consider evolution operators admitting the following expansion in powers of~$\DT$ (see~\eqref{eqn:semigroup-expansion}):
\begin{equation*}
P_{\DT} = I + \DT\mathcal{A}_1 + \ldots + \DT^{p+1}\mathcal{A}_{p+1}
+ \DT^{p+2} \mathcal{R}_{p, \DT},
\end{equation*}
where the explicit expression of the operators $\mathcal{A}_n$ for $n = 1, \dots, p + 1$ and $\mathcal{R}_{p, \DT}$ can be identified systematically by Taylor expansions. 
Note first that $\widehat{\theta}_{\DT}$ satisfies
\begin{equation}\label{eqn:discrete-Poisson-equation-projected}
    \Pi\left[\frac{I - P_{\DT}}{\DT}\right]\Pi \widehat{\theta}_{\DT} = \theta - \mu(\theta).
\end{equation}
In order to find an approximation to $\widehat{\theta}_{\DT}$, we expand the operator $\DT^{-1}\Pi[I - P_{\DT}]\Pi$ in powers of $\DT$ as
\[
\Pi\left[\frac{I - P_{\DT}}{\DT}\right]\Pi = -\left(\widetilde{\mathcal{A}}_1 + \DT\widetilde{\mathcal{A}}_2 + \ldots + \DT^p \widetilde{\mathcal{A}}_{p+1}\right) - \DT^{p+1}\Pi \mathcal{R}_{p, \DT} \Pi, 
\]
where $\widetilde{\mathcal{A}}_n = \Pi \mathcal{A}_n \Pi$ for $n = 1, 2, \ldots, p+1$. Defining $\mathcal{B} = \widetilde{\mathcal{A}}_2 + \DT\widetilde{\mathcal{A}}_3 + \ldots + \DT^{p-1} \widetilde{\mathcal{A}}_{p+1}$, it holds
\[
\widetilde{\mathcal{A}}_1 + \DT\widetilde{\mathcal{A}}_2 + \ldots + \DT^{p} \widetilde{\mathcal{A}}_{p+1} = \widetilde{\mathcal{A}}_1 + \DT \mathcal{B} .
\]
Recalling that $\widetilde{\mathcal{A}}_1^{-1}$ is by assumption well defined from $\mathcal{S}_0$ to $\mathcal{S}_0$, the formal series expansion of the inverse of $\widetilde{\mathcal{A}}_1 + \DT \mathcal{B}$ is
\[
\widetilde{\mathcal{A}}_1^{-1} - \DT \widetilde{\mathcal{A}}_1^{-1}\mathcal{B} \widetilde{\mathcal{A}}_1^{-1} + \DT^2 \widetilde{\mathcal{A}}_1^{-1}\mathcal{B}\widetilde{\mathcal{A}}_1^{-1}\mathcal{B}\widetilde{\mathcal{A}}_1^{-1} + \ldots.
\]
Truncating the above formal series expansion up terms involving at most~$p$ instances of~$\mathcal{B}$, we end up with the operator
\[
\widetilde{Q}_{\DT} \triangleq \widetilde{\mathcal{A}}_1^{-1}\sum_{n=0}^{p} (-\DT)^n \left(\mathcal{B}\widetilde{\mathcal{A}}_1^{-1}\right)^n,
\]
which is such that the following equality holds on~$\mathcal{S}_0$:
\[
\left(\widetilde{\mathcal{A}}_1 + \DT \mathcal{B}\right)\widetilde{Q}_{\DT} = \Pi + (-1)^p \DT^{p+1} \left(\mathcal{B}\widetilde{\mathcal{A}}_1^{-1}\right)^{p+1}.
\]
We are now ready to define the approximate inverse operator $Q_{\DT}$ by expanding $\widetilde{Q}_{\DT}$ and keeping terms up to order~$\DT^p$, i.e., 
\[
Q_{\DT} \triangleq \widetilde{\mathcal{A}}_1^{-1} - \DT \widetilde{\mathcal{A}}_1^{-1} \widetilde{\mathcal{A}}_2 \widetilde{\mathcal{A}}_1^{-1}
+ \DT^2 \left(\widetilde{\mathcal{A}}_1^{-1} \widetilde{\mathcal{A}}_2 \widetilde{\mathcal{A}}_1^{-1}\widetilde{\mathcal{A}}_2 \widetilde{\mathcal{A}}_1^{-1} - \widetilde{\mathcal{A}}_1^{-1}\widetilde{\mathcal{A}}_3 \widetilde{\mathcal{A}}_1^{-1}\right) + \ldots + \DT^{p-1}\mathcal{Q}_{p-1} + \DT^{p} \mathcal{Q}_{p},
\] 
where $\mathcal{Q}_{n}$ for $n = 1, \ldots, p$ are operators mapping~$\mathcal{S}_0$ to~$\mathcal{S}_0$. Note that the approximate inverse operator $Q_{\DT}$ leaves $\mathcal{S}_0$ invariant. Finally, we define the approximate discrete Poisson solution by 
\begin{equation}\label{eqn:approx-inverse-solution}
    \widetilde{\theta}_{\DT} = -Q_{\DT}(\theta - \mu(\theta)). 
\end{equation}
The function~$\widetilde{\theta}_{\DT}$ indeed belongs to~$\mathcal{S}_0$. Moreover, it can be readily verified that 
\begin{equation*}
\Pi\left[\frac{I - P_{\DT}}{\DT}\right]\Pi \widetilde{\theta}_{\DT} = \theta - \mu(\theta) + \DT^{p+1} \phi_{\DT, p, \theta},
\end{equation*}
for some function $\phi_{\DT, p, \theta} \in \mathcal{S}_0$ that is uniformly bounded with respect to $\DT$ (in the sense of~\eqref{eq:unif_bound_remainder_approx_inverse}).

Finally, to obtain the estimates on $\widetilde{\theta}_{\DT}-\widehat{\theta}$, we note that, from the definition~\eqref{eqn:continuous-Poisson-equation} and~\eqref{eqn:approx-inverse-solution}, the following equality holds when $\mathcal{A}_1 = \mathcal{L}$:
\begin{equation}\label{eqn:diff-discrete-continuous-Poisson-pointwise}
\begin{split}
\widetilde{\theta}_{\DT}-\widehat{\theta} =  
\DT &\left[ \widetilde{\mathcal{A}}_1^{-1} \widetilde{\mathcal{A}}_2 \widetilde{\mathcal{A}}_1^{-1} - \DT \left(\widetilde{\mathcal{A}}_1^{-1} \widetilde{\mathcal{A}}_2 \widetilde{\mathcal{A}}_1^{-1}\widetilde{\mathcal{A}}_2 \widetilde{\mathcal{A}}_1^{-1} - \widetilde{\mathcal{A}}_1^{-1}\widetilde{\mathcal{A}}_3 \widetilde{\mathcal{A}}_1^{-1}\right) + \ldots \right.\\
&\quad \left.- \DT^{p-2}\mathcal{Q}_{p-1} - \DT^{p-1} \mathcal{Q}_{p}\right](\theta - \mu(\theta)).
\end{split}
\end{equation}
Theorem~\ref{thm:approx-inverse} follows immediately from the above discussion.

\subsection{An estimate of the elementary term}
\label{sec:estimate_elmentary}

Given an observable $\phi \in B^{\infty}$, we refer to the average $N^{-1}\sum_{n=0}^{N-1} \mathbb{E}_{\DT}\{\phi_n Z_N\}$ as an ``elementary term'' since quantities of this form are of fundamental importance in establishing the results in this work. Note that the LR sensitivity estimator is of this form. In this section, we establish bounds on such elementary terms, which allow to estimate remainders when performing Taylor expansions as in Theorem~\ref{thm:first-order-error}. The following bound is crude, but its strength is that it is uniform with respect to the test function. This is crucial since remainder functions, altough uniformly bounded in $B^\infty$, depend on the time step~$h$.

\begin{lemma}\label{lemma:elementary-coarse-estimate}
  Consider a discrete martingale $Z_N$, and assume that there exists a constant $K \in \mathbb{R}_+$ such that $\eta_n = Z_{n+1} - Z_n$ satisfies $\mathbb{E}_\DT(\eta_n^2) \leq K\DT$. Then there exist $\DT^* > 0$ and $C \in \mathbb{R}_+$ such that, for any $\DT \in (0, \DT^*]$ and any $\phi \in B^{\infty}$,
\[
\left| \frac{1}{N}\sum_{n=0}^{N-1} \mathbb{E}_{\DT}\{\phi_n Z_N\} \right| \leq \frac{C}{\sqrt{\DT}} \|\phi\|_{B^\infty}.
\]
\end{lemma}

This estimate can be used with the martingale increments obtained from~\eqref{eqn:discrete-Z} and~\eqref{eqn:discrete-Z_2-scalar} of the schemes we consider in this work.

\begin{proof}
  Throughout the proof, we denote by $C>0$ a generic constant which may change from line to line. Note first that $\mathbb{E}_{\DT}\{\phi_n Z_N\} = \mathbb{E}_{\DT}\{[\phi_n-\mu_{\DT}(\phi)] Z_N\}$. We use the discrete Poisson equation 
\[
  \left[\frac{I - P_{\DT}}{\DT}\right]\widehat{\phi}_{\DT} = \phi - \mu_{\DT}(\phi).
\]
By~\eqref{eq:bound_inverse_Ph} there exist $h^*$ and $R>0$ such that $\|\widehat{\phi}_{\DT}\|_{B^\infty} \leq R \|\phi\|_{B^\infty}$ for $h \in (0,h^*]$. 
We next rewrite the left-hand side of the desired inequality as
\begin{equation}\label{eqn:bdd-fundamental-term}
\begin{split}
&\frac{1}{N\DT}\sum_{n = 0}^{N-1} \mathbb{E}_{\DT}\left\{\left(\widehat{\phi}_{\DT, n} - P_{\DT}\widehat{\phi}_{\DT, n}\right) Z_N\right\}\\
& \qquad = \frac{1}{N\DT}\sum_{n = 0}^{N-1} \mathbb{E}_{\DT}\left\{\left(\widehat{\phi}_{\DT, n+1} - P_{\DT}\widehat{\phi}_{\DT, n}\right) Z_N\right\} -\frac{1}{N\DT}\mathbb{E}_{\DT}\left\{\left(\widehat{\phi}_{\DT, N} - \widehat{\phi}_{\DT, 0}\right) Z_N\right\}.
\end{split}
\end{equation}
For convenience, we denote the martingale differences by $\xi_n = \widehat{\phi}_{\DT, n+1} - P_{\DT}\widehat{\phi}_{\DT, n}$, and hence
\[
\frac{1}{N\DT}\sum_{n = 0}^{N-1} \mathbb{E}_{\DT}\left\{\left(\widehat{\phi}_{\DT, n+1} - P_{\DT}\widehat{\phi}_{\DT, n}\right) Z_N\right\} = \frac{1}{N\DT}\sum_{n = 0}^{N-1} \mathbb{E}_{\DT}\{\xi_n \eta_n\}.
\]
Note that by the Cauchy--Schwarz inequality, 
\[
\left|\mathbb{E}_{\DT}\{\xi_n \eta_n\}\right| \leq \mathbb{E}_{\DT}\{\xi_n^2\}^{1/2} \mathbb{E}_{\DT}\{\eta_n^2\}^{1/2} \leq \sqrt{K\DT} \mathbb{E}_{\DT}\{\xi_n^2\}^{1/2}. 
\]
Since $|\xi_n| \leq 2 \|\widehat{\phi}_{\DT}\|_{B^\infty} \leq 2R \|\phi\|_{B^\infty}$, we can conclude that
\begin{equation}\label{eqn:fundamental-coarse-1}
\left| \frac{1}{N\DT}\sum_{n = 0}^{N-1} \mathbb{E}_{\DT}\{\xi_n \eta_n\} \right| \leq \frac{C}{\sqrt{\DT}} \|\phi\|_{B^\infty}.
\end{equation}
Now, for the second term on the right hand side of~\eqref{eqn:bdd-fundamental-term}, 
the Cauchy--Schwarz inequality gives
\begin{equation}\label{eqn:fundamental-coarse-2}
  \begin{aligned}
  \left| \frac{1}{N\DT}\mathbb{E}_{\DT}\left\{ \left(\widehat{\phi}_{\DT, N} - \widehat{\phi}_{\DT, 0}\right) Z_N\right\} \right|
  & \leq \frac{1}{N\DT}\mathbb{E}_{\DT}\left\{\left(\widehat{\phi}_{\DT, N} - \widehat{\phi}_{\DT, 0}\right)^2\right\}^{1/2} \mathbb{E}_{\DT}\{Z_N^2\}^{1/2} \\
 & \leq \frac{C\|\phi\|_{B^\infty}}{N\DT}\left[\sum_{n=0}^{N-1}\mathbb{E}_{\DT}\{\eta_n^2\}\right]^{1/2}
  \leq \frac{C \|\phi\|_{B^\infty}}{\sqrt{N\DT}}.
\end{aligned}
\end{equation}
Finally, the result follows by combining the estimates~\eqref{eqn:fundamental-coarse-1} and~\eqref{eqn:fundamental-coarse-2}.
\end{proof}

\appendix

\section{Derivation of the modified martingale for the additive noise in the multi-dimensional case}
\label{sec:appendix}

We have shown in the proof of Theorem~\ref{thm:second-order-error} that the modified martingale \eqref{eqn:discrete-Z_2-scalar} leads to the correct second order CLR scheme in the scalar setting.
For the sake of completeness we provide a detailed  algebraic calculations to justify that, under the additional assumption of the additive noise (i.e., $\sigma(x)$ is independent of~$x$), the formula \eqref{eqn:discrete-Z_2-vector} leads to the second order CLR scheme in the multi-dimensional setting.

Similar to the proof of Theorem~\ref{thm:second-order-error} the multivariate expansion (in $\DT$) of the Poisson solution $\widetilde{\theta}_{\DT, n+1}$ reads 
\begin{equation*}
\begin{split}
\widetilde{\theta}_{\DT, n+1} = 
\widetilde{\theta}_{\DT, n}
&+ D^1 \widetilde{\theta}_{\DT, n}^\TR  \Phi_{\DT, n}
+ \frac{1}{2} D^2 \widetilde{\theta}_{\DT, n} : \Phi_{\DT, n}^{\otimes 2}
+ \frac{1}{6}D^3 \widetilde{\theta}_{\DT, n}  : \Phi_{\DT, n}^{\otimes 3} + r_{\DT, \theta, n}, \end{split}
\end{equation*}
with the remainder
\[
r_{\DT, \theta, n} = \left(\frac{1}{6}\int_0^1 u^3  D^4\widetilde{\theta}_{\DT}(X_n + u \Phi_{\DT, n})  \,du\right) : \Phi_{\DT, n}^{\otimes 4},
\]
where $D^k$ denotes the $k$-th order differential: for $v_1,\dots,v_k \in \mathbb{R}^d$,
\[
D^k f(X) : (v_1 \otimes \dots \otimes v_k) = \sum_{i_1 + \dots + i_k = d} \frac{\partial^k f}{\partial_{x_1}^{i_1} \dots \partial_{x_d}^{i_d}}(X) v_1^{i_1} \dots v_d^{i_d}.
\]
In the presentation of formulas below we use a matrix notation in which the  gradient 
$\nabla \widetilde\theta$ is viewed as a column vector with components $\partial_{x_i} \widetilde\theta$ and the second order differential $D^2 \widetilde\theta$ is represented by the Hessian matrix $\nabla^2\widetilde\theta$ of the second derivatives $\partial^2_{x_i,x_j}\widetilde\theta$.
We recall the increment function $\Phi_{\DT, n}$  for the second order discretization as defined in \eqref{eqn:second-order-scheme}, the corresponding 
induced semigroup $P_{\DT}$ and the modified martingale $Y_n$ as defined in \eqref{eqn:discrete-Z_2-scalar}.
After expanding $\widetilde{\theta}_{\DT, n+1} - P_{\DT}\widetilde{\theta}_{\DT, n}$ in powers of~$\DT^{1/2}$, 
the resulting terms of order $h^{1/2}$ and order $h^{3/2}$ are
\begin{equation}\label{eqn:X-increment-1}
\nabla \widetilde{\theta}_{\DT, n}^\TR  \sigma \Delta W_n
\end{equation}
and 
\begin{equation}\label{eqn:X-increment-2}
\begin{split}
&\frac{1}{2}\nabla \widetilde{\theta}_{\DT, n}^\TR (\mathcal{K} b_n) 
\Delta W_n \DT + (\sigma \Delta W_n)^{\TR}   \nabla^2 \widetilde{\theta}_{\DT, n} b_n h +
\frac{1}{6} D^3 \widetilde{\theta}_{\DT, n}  (\sigma \Delta W_n)^{\otimes 3},
\end{split}
\end{equation}
respectively, where $\nabla b_n = [\nabla b_n^1 , \ldots, \nabla b_n^d] \in \mathbb{R}^{d \times d}$.
Similarly, for the increment $Y_{n+1} - Y_n$, the terms of order $h^{1/2}$ and order $h^{3/2}$ are
\begin{equation}\label{eqn:Y-increment-1}
(\sigma^{-1} F_n)^\TR \Delta W_n
\end{equation}
and 
\begin{equation}\label{eqn:Y-increment-2}
\frac{1}{2}(\sigma^{-1} F_n)^{\TR} (\mathcal{K} b_n)^{\TR} \sigma^{-\TR} \Delta W_n \DT,
\end{equation}
respectively.

Next, we expand the product $\DT^{-1}\mathbb{E}_{\DT}\{(\widetilde{\theta}_{\DT, n+1} - P_{\DT}\widetilde{\theta}_{\DT, n})(Y_{n+1} - Y_n) \}$ and compute the terms of order $1$
and~$\DT$ as follows.
By multiplying $\eqref{eqn:X-increment-1}$ and $\eqref{eqn:Y-increment-1}$ together, we
obtain the term of order $1$
\begin{equation}\label{eqn:order-zero}
\mathbb{E}_h \left[ \nabla \widetilde{\theta}_{\DT, n}^\TR  \sigma \Delta W_n  \Delta W_n^\TR   (\sigma^{-1} F_n)\right]  = \mathbb{E}_h \left[  \nabla \widetilde{\theta}_{\DT, n}^\TR F_n \right],
\end{equation}
which is consistent with the univariate case. 
The computation of the order $h$ terms is more involved.
First the product of \eqref{eqn:X-increment-1} and \eqref{eqn:Y-increment-2} leads to 
\begin{equation}\label{eqn:order-one-1}
\begin{split}
&\frac{h}{2}\mathbb{E}_h\left[  (\sigma^{-1} F_n)^\TR \left(\mathcal{K} b_n\right)^{\TR} \sigma^{-\TR} \Delta W_n  \Delta W_n^{\TR} \sigma^{\TR} \nabla \widetilde{\theta}_{\DT, n}  \right] =  
\frac{h}{2}\mathbb{E}_h\left[ (\sigma^{-1} F_n)^\TR \left(\mathcal{K} b_n\right)^{\TR}  \nabla \widetilde{\theta}_{\DT, n} \right].
\end{split}
\end{equation}
There are additional terms of order $h$ coming from the product of \eqref{eqn:X-increment-2} and \eqref{eqn:Y-increment-1}.
The multiplication of the first term of \eqref{eqn:X-increment-2} and \eqref{eqn:Y-increment-1} leads to
\begin{equation}\label{eqn:order-one-2}
\begin{split}
&\frac{h}{2}\mathbb{E}_h
\left[ 
(\sigma^{-1} F_n)^{\TR} \Delta W_n \Delta W_n^{\TR}
(\mathcal{K} b_n)^{\TR} 
\nabla \widetilde{\theta}_{\DT, n}
\right] =
\frac{h}{2}\mathbb{E}_h
\left[ 
(\sigma^{-1} F_n)^{\TR}
(\mathcal{K} b_n)^{\TR} 
\nabla \widetilde{\theta}_{\DT, n}
\right].
\end{split}
\end{equation}
Note that the sum of \eqref{eqn:order-one-1} and \eqref{eqn:order-one-2} is
\begin{equation}\label{eqn:order-one-1+2}
h\mathbb{E}_h \left[ (\sigma_n^{-1} F_n)^{\TR}
(\mathcal{K} b_n)^{\TR} 
\nabla \widetilde{\theta}_{\DT, n}    \right]
=
h\mathbb{E}_h\left[
(\nabla b_n \nabla \widetilde{\theta}_{\DT, n})^{\TR}
F_n
\right].
\end{equation}
Similarly, multiplying the second term of \eqref{eqn:X-increment-2} by \eqref{eqn:Y-increment-1} leads to 
\begin{equation}\label{eqn:order-one-3}
h \mathbb{E}\left[ F_n^{\TR} \sigma^{-\TR} \Delta W_n \Delta W_n^{\TR} \sigma^\TR   \nabla^2 \widetilde{\theta}_{\DT, n} b_n \right] = h \mathbb{E}_h \left[  \left(\nabla^2 \widetilde{\theta}_{\DT, n} b_n \right)^{\TR}  F_n\right].
\end{equation}
It only remains to compute the product of the last term of \eqref{eqn:X-increment-2} and   
\eqref{eqn:Y-increment-1}.
To this end, note that
\begin{equation}\label{eqn:order-one-4}
\begin{split}
&\frac{\DT}{6}\mathbb{E}_h 
\left[
D^3 \widetilde{\theta}_{\DT, n} : (\sigma_n\Delta W_n)^{\otimes 3} (\sigma_n^{-T}\Delta W_n)^\TR F_n
\right]\\
& \qquad =
\frac{\DT}{6}\sum_{i,j,k,l=1}^d\mathbb{E}_{\DT}
\left[
\partial_{x_i,x_j,x_k}^3\widetilde{\theta}_{\DT, n} 
(\sigma \Delta W_n)^i
(\sigma \Delta W_n)^j
(\sigma \Delta W_n)^k
(\sigma^{-\TR}\Delta W_n)^l F_n^l
\right]\\
& \qquad = 
\frac{\DT}{2}\sum_{i,j,k,l=1}^d
\mathbb{E}_{\DT}\left[\partial_{x_i,x_j,x_k}^3\widetilde{\theta}_{\DT, n}
\sum_{\alpha=1}^d \sigma^{i\alpha} \sigma^{j\alpha}
\sum_{\beta=1}^d \sigma^{k\beta} (\sigma^{-1})^{\beta l} F_n^l \right]\\
& \qquad =
\frac{\DT}{2}\sum_{i,j,k=1}^d\mathbb{E}_{\DT}\left[
\left(\sigma \sigma^{\TR}\right)^{ij}\partial_{x_i,x_j,x_k}^3\widetilde{\theta}_{\DT, n}F_n^k\right],
\end{split}
\end{equation}
where we have used the fact that (with the usual definition of the Kronecker symbol $\delta_{\alpha\beta}$)

\begin{equation*}
 \mathbb{E}_{\Delta W_n}[\Delta W_n^{\alpha}\Delta W_n^{\beta} \Delta W_n^{\gamma}\Delta W_n^{\delta}] = \delta_{\alpha\beta}\delta_{\gamma\delta} +
 \delta_{\alpha\gamma}\delta_{\beta\delta} +
 \delta_{\alpha\delta}\delta_{\beta\gamma}
\end{equation*}

Finally, combining \eqref{eqn:order-one-1+2} to \eqref{eqn:order-one-4} we obtain (recalling that~$\sigma$ is constant)
\[
\DT \mathbb{E}_h\left[
F_n^\TR \nabla b_n \nabla \widetilde{\theta}_{\DT, n}
F_n
+
F_n^\TR \nabla^2 \widetilde{\theta}_{\DT, n} b_n 
+
\frac12 F_n^\TR \nabla \left( \sum_{i,j=1}^d \left(\sigma\sigma^\TR\right)^{ij} \partial_{x_i,x_j}^2 \widetilde{\theta}_{\DT, n} \right)
\right]
=
\DT \mathbb{E}_{\DT}\left[ F_n^\TR\nabla\mathcal{L}\widetilde{\theta}_{\DT, n} \right],
\]
which leads to the order $\DT$ correction term $-\DT \mathbb{E}_{\DT}[F_n^\TR \nabla \theta_n]$ by the same argument as in the proof of Theorem~\ref{thm:second-order-error}.

\section*{Acknowledgement}
We thank David Aristoff for useful discussions. The research of T.W. was sponsored by the CCDC Army Research Laboratory and was accomplished under Cooperative Agreement Number W911NF-16-2-0190. The work of T.W. and P.P. was supported in part by the DARPA project W911NF-15-2-0122, while the work of G.S. was funded by the Agence Nationale de la Recherche, under grant ANR-14-CE23-0012 (COSMOS), and by the European Research Council (ERC) under the European Union's Horizon 2020 research and innovation programme (grant agreement No 810367). G.S. also benefited from the scientific environment of the Laboratoire International Associ\'e between the Centre National de la Recherche Scientifique and the University of Illinois at Urbana-Champaign.

\section*{Disclaimer}
The views and conclusions contained in this document are those of the authors and should not be interpreted as representing the official policies, either expressed or implied, of the Army Research Laboratory or the U.S. Government. The U.S. Government is authorized to reproduce and distribute reprints for Government purposes notwithstanding any copyright notation herein.  

\bibliographystyle{siamplain}
\bibliography{references}
\end{document}